\documentclass[reqno]{amsart}
\usepackage{amsmath}
\usepackage{amsthm}
\usepackage{mathtools}
\usepackage{url}
\usepackage{amssymb}
\usepackage[pdftex]{graphicx}

\newtheorem{theorem}{Theorem}[section]
\newtheorem{lemma}[theorem]{Lemma}

\theoremstyle{definition}

\theoremstyle{remark}
\newtheorem{remark}[theorem]{Remark}

\numberwithin{equation}{section}

\newcommand{\rd}{\mathrm{d}}
\newcommand{\RR}{\mathbb{R}}
\newcommand{\CC}{\mathbb{C}}

\newcommand{\ud}[2]{u^{\left( #1 \right)}_{ #2 }}

\newcommand{\vd}[2]{v^{\left( #1 \right)}_{ #2 }}
\newcommand{\Vd}[2]{V^{\left( #1 \right)}_{ #2 }}
\newcommand{\tVd}[2]{\tilde{V}^{\left( #1 \right)}_{ #2 }}

\newcommand{\eEd}[2]{e^{\left( #1 \right)}_{E #2 }}
\newcommand{\EEd}[2]{e^{\left( #1 \right)}_{E, #2 }}

\newcommand{\eNd}[2]{e^{\left( #1 \right)}_{N #2 }}
\newcommand{\ENd}[2]{e^{\left( #1 \right)}_{N, #2 }}
\newcommand{\eVd}[2]{e^{\left( #1 \right)}_{V #2 }}
\newcommand{\EVd}[2]{e^{\left( #1 \right)}_{V, #2 }}
\newcommand{\Ed}[2]{E^{\left( #1 \right)}_{ #2 }}
\newcommand{\tEd}[2]{\tilde{E}^{\left( #1 \right)}_{ #2 }}

\newcommand{\Nd}[2]{N^{\left( #1 \right)}_{ #2 }}
\newcommand{\tNd}[2]{\tilde{N}^{\left( #1 \right)}_{ #2 }}

\newcommand{\taEd}[2]{\tau^{\left( #1 \right)}_{E #2 }}
\newcommand{\TaEd}[2]{\tau^{\left( #1 \right)}_{E, #2 }}

\newcommand{\taNd}[2]{\tau^{\left( #1 \right)}_{N #2 }}
\newcommand{\TaNd}[2]{\tau^{\left( #1 \right)}_{N, #2 }}
\newcommand{\taVd}[2]{\tau^{\left( #1 \right)}_{V #2 }}
\newcommand{\TaVd}[2]{\tau^{\left( #1 \right)}_{V, #2 }}
\newcommand{\fd}{\delta^+}
\newcommand{\bd}{\delta^-}

\newcommand{\cdd}{\delta^{\left\langle 2 \right\rangle}}

\newcommand{\fa}{\mu^+}
\newcommand{\ba}{\mu^-}
\newcommand{\ca}{\mu^{\left\langle 1 \right\rangle}}

\newcommand{\Dx}{\Delta x}
\newcommand{\Dt}{\Delta t}
\newcommand{\dt}{\mathrm{d} t}

\newcommand{\dx}{\mathrm{d} x}

\newcommand{\sn}{\mathrm{sn}}
\newcommand{\cn}{\mathrm{cn}}
\newcommand{\dn}{\mathrm{dn}}
\newcommand{\Emax}{E_\mathrm{max}}

\begin{document}

\title[Analysis and comparison of energy-conservative schemes for Zakharov]{Mathematical analysis and numerical comparison of energy-conservative schemes for the Zakharov equations}


\author{Shuto Kawai}
\address{Department of Mathematical Informatics, Graduate School of Information Science and Technology, University of Tokyo, 7-3-1, Bunkyo-ku, Tokyo 113-8656, Japan}
\curraddr{}
\email{kawai-sk@g.ecc.u-tokyo.ac.jp}
\thanks{}

\author{Shun Sato}
\address{}
\curraddr{}
\email{}
\thanks{}

\author{Takayasu Matsuo}



\date{}

\dedicatory{}

\begin{abstract}
  Furihata and Matsuo proposed in 2010 an energy-conserving scheme for the Zakharov equations, as an application of the discrete variational derivative method (DVDM).
  This scheme is distinguished from conventional methods (in particular the one devised by Glassey in 1992) in that the invariants are consistent with respect to time, but it has not been sufficiently studied both theoretically and numerically.
  In this study, we theoretically prove the solvability under the loosest possible assumptions.
  We also prove the convergence of this DVDM scheme by improving the argument by Glassey.
  Furthermore, we perform intensive numerical experiments for comparing the above two schemes.
  It is found that the DVDM scheme is superior in terms of accuracy, but since it is fully-implicit, the linearly-implicit Glassey scheme is better for practical efficiency.
  In addition, we proposed a way to choose a solution for the first step that would allow Glassey's scheme to work more efficiently.
\end{abstract}

\maketitle

\section{Introduction}

In this paper, for $E: \RR \times \RR \to \CC$ and $N: \RR \times \RR \to \RR$, we consider the initial value problem of the Zakharov equations:
\begin{align}\label{eq_Zak}
  \begin{cases}
    \mathrm{i} E_t + E_{xx} &= NE \\
    N_{tt} - N_{xx} &= \left( \left\vert E \right\vert^2 \right)_{xx}
  \end{cases} 
\end{align}
with the initial conditions
\begin{align*}
  E(0,x) = E^0(x),\quad N(0,x) = N^0(x),\quad N_t(0,x) = N^1(x) \qquad \left(x \in \mathbb{R}\right),
\end{align*}
and the periodic boundary conditions
\begin{align}\label{eq_period}
  E(t, x+L) = E(t,x),\quad N(t, x+L) = N(t,x) \qquad \left(t>0,\, x \in \mathbb{R}\right)
\end{align}
for some period $L > 0$.

It is well known that \eqref{eq_Zak} possesses the following invariants:
\begin{align}
  \mathcal{N} (t) &\coloneqq \int_0^L \left\vert E(t,x) \right\vert^2 \dx = \mathcal{N} (0),\\
  \mathcal{E} (t) &\coloneqq \int_0^L \left[ \left\vert E_x(t,x) \right\vert^2 + \frac{1}{2}\left( N(t,x)^2 + V_x(t,x)^2 \right) + N(t,x) \left\vert E(t,x) \right\vert^2 \right] \dx \\
  &= \mathcal{E} (0),\nonumber
\end{align}
where $V(t,x)$ is uniquely defined by
\begin{align}\label{eq_Vdef}
  N_t =V_{xx},\quad V_t =N+\vert E \vert^2,\quad V(0,0) = 0.
\end{align}
$\mathcal{N}(t)$ and $\mathcal{E}(t)$ are called ``norm'' and ``energy''. For $V$ and $V_x$ to also satisfy the periodic boundary conditions
\begin{align}\label{eq_period2}
  V(t, x+L) = V(t,x),\quad V_x(t, x+L) = V_x(t,x), \qquad \left(t>0,\, x \in \mathbb{R}\right)
\end{align}
we need to assume that
\begin{align}\label{eq_assumption}
  \int_0^L N^1(x) \dx = 0.
\end{align}

The Zakharov equations describe the propagation of Langmuir waves in a plasma. Here, $E(t, x) \in \mathbb{C}$ is the envelope of the high-frequency electric field and $N (t, x) \in \mathbb{R}$ is the deviation of the ion density from equilibrium value.
Since its introduction by Zakharov~\cite{Zakharov}, the Zakharov equations have been widely recognized as a general model governing the interaction of dispersive and non-dispersive waves.
It is now applied to a wide variety of physical problems, including theories of molecular dynamics and fluid mechanics.
Theoretical properties of Zakharov systems, such as the existence of regularity of solutions, uniqueness, and collision of solutions, are shown by~\cite{Bourgain,Colliander,Compaan} and others.

Numerical methods for the Zakharov equations and related equations have been studied using various methods, such as the time-splitting spectral (TSSP) method~\cite{Bao1, Jin}, the pseudospectral method~\cite{Bao2}, the multisymplectic method~\cite{Wang}, the local discontinuous Galerkin methods \cite{Xia}, the Legendre spectral method~\cite{Ji}, and Jacobi pseudospectral approximation~\cite{Bhrawy}.
Among them, several structure-preserving finite difference schemes have also been proposed, for example, \cite{Chang, Glassey, Pan, Zhou}, which have been shown to be useful through numerical experiments and theoretical proofs. However, their corresponding discrete conserved quantities are defined with multiple numerical solutions at several different time steps, and furthermore they are not symmetric in time. This is in some sense inconsistent and not desirable, since the original continuous conserved quantities are defined on a single time point $t$. This distortion may damage the advantage originally expected to such structure-preserving methods.
An energy-preserving scheme without such distortion was first presented in a textbook of a framework for designing structure-preserving schemes, called the discrete variational derivative method (DVDM)~\cite{DVDM}.
This scheme is, however, only defined as one of the illustrating examples of the framework, and no numerical experiments or mathematical analysis have been performed there and so far.
In addition, although this scheme is surely favorable in that it has consistent discrete conserved quantities defined on a single time step, it is obviously computationally expensive due to its full-implicitness.
Thus it is not at all clear whether the scheme is in fact advantageous compared to cheaper (but distorted) schemes, such as the linearly-implicit difference scheme proposed by Glassey~\cite{Glassey}.

In view of the above background, we have two aims in the present paper.
First, we give a rigorous mathematical analysis of the DVDM scheme.
Next, we provide intensive numerical comparisons with Glassey's scheme, which reveals actual performance of these structure-preserving schemes.

This paper is organized as follows.
In Section~2, notations and lemmas used in the following sections are introduced.
In Section~3, the results of mathematical analysis are described.
In particular, the convergence is proved in a way that improves on the arguments of the previous study.
In Section~4, numerical experiments are conducted to compare the performance of the DVDM scheme with Glassey's scheme.
Finally, our conclusions are summarized in Section~5.


\section{Preliminaries}


\subsection{Notations and discrete symbols}

In order to introduce discrete symbols, operators and norms (shown in~\cite{DVDM}, for example),
we use the symbol $ \ud{m}{k}$ $( m=0,\dots, M$; $k \in \mathbb{Z}$) to represent $E,$ $N$ and $V$, where $ \Dt$ ($\coloneqq T/M$ for some $T\in\mathbb{R}_+$) and $ \Delta x \, (\coloneqq L/K)$ are the temporal and spatial mesh sizes, respectively.
Here, we assume the discrete periodic boundary condition $ \ud{m}{k+K} = \ud{m}{k} \left ( k \in \mathbb{Z} \right) $.
We also use the notation $ \ud{m}{} \coloneqq \left( \ud{m}{1} , \dots ,\ud{m}{K} \right)^{\top} $.\\

The spatial forward and backward difference operators are defined as follows:
\begin{align*}
  \fd_x \ud{m}{k} \coloneqq \frac{ \ud{m}{k+1} - \ud{m}{k} }{\Dx}, \quad \bd_x \ud{m}{k} \coloneqq \frac{\ud{m}{k} - \ud{m}{k-1}}{\Dx}.
\end{align*}
To keep our presentation simple, we also use the notation $\fd_x \ud{m}{} \coloneqq \left( \fd_x \ud{m}{k} \right)_k^{\top}$. This also applies to other discrete operators below.

The spatial second order central difference operator is defined as follows:
\begin{align*}
  \cdd_x \ud{m}{k} \coloneqq \frac{ \ud{m}{k+1} - 2\ud{m}{k} + \ud{m}{k-1} }{(\Dx)^2}.
\end{align*}
The spatial forward average operator is defined as follows:
\begin{align*}
  \fa_x \ud{m}{k} \coloneqq \frac{ \ud{m}{k+1} + \ud{m}{k} }{2}.
\end{align*}
The temporal forward and backward difference operators are defined as follows:
\begin{align*}
  \fd_t \ud{m}{k} \coloneqq \frac{\ud{m+1}{k} - \ud{m}{k}}{\Delta t}, \quad \bd_t \ud{m}{k} \coloneqq \frac{ \ud{m}{k} - \ud{m-1}{k} }{\Dt}.
\end{align*}
The temporal forward, backward and central average operators are defined as follows:
\begin{align*}
  \fa_t \ud{m}{k} &\coloneqq \frac{ \ud{m+1}{k} + \ud{m}{k} }{2}, \quad \ba_t \ud{m}{k} \coloneqq \frac{ \ud{m}{k} + \ud{m-1}{k} }{2}, \\
  \ca_t \ud{m}{k} &\coloneqq \frac{ \ud{m+1}{k} + \ud{m-1}{k} }{2}.
\end{align*}

We define the discrete Lebesgue space $L_K^p \, \left(1 \leq p \leq \infty \right)$ as the pair $(\CC^K, \left\| \cdot \right\|_p )$, where the norm is defined as
\begin{align*}
  \|v\|_p \coloneqq \left(\sum_{k=1}^K \left\vert v_k \right\vert^p \Dx \right)^{\frac{1}{p}} (1 \leq p < \infty),\quad \left\| v \right\|_\infty \coloneqq \max_{k\in\mathbb{Z}}\left\vert v_k \right\vert.
\end{align*}
For $L_K^2$, the associated inner product $\left\langle \cdot, \cdot \right\rangle$ is defined as
\begin{align*}
  \left\langle v, w\right\rangle \coloneqq \sum_{k=1}^K v_k \bar{w}_k \Dx.
\end{align*}
For simplicity, we abbreviate $\|\cdot\|_2$ by $\|\cdot\|$ hereafter.

Then, several basic properties hold as follows (their proofs are straightforward from the definition):


\begin{lemma}\label{lem_ope}
  The following properties hold:
  \begin{enumerate}
    \item All the aforementioned operators commute with each other;
    \item $\cdd_x = \fd_x \bd_x;$
    \item $\fd_x ( v_k w_k ) = ( \fd_x v_k ) w_{k+1} + v_k (\fd_x w_k);$
    \item (skew-symmetry) $\left\langle \fd_x v, w \right\rangle = -\left\langle v,\bd_x w \right\rangle;$
    \item $2\left\langle v, w \right\rangle \leq \| v \|^2 + \| w \|^2;$
    \item $\left\| v + w  \right\|^2 \leq 2 (\| v \|^2 + \| w \|^2);$
    \item $\fd_t \left\| z^{(m)} \right\|^2 = 2 \mathrm{Re} \left\langle \fd_t z^{(m)}, \fa_t z^{(m)} \right\rangle$.
  \end{enumerate}
\end{lemma}


\begin{lemma}[Discrete Sobolev Lemma; Lemma 3.2 in \cite{DVDM}]\label{lem_Sobolev}
  For any $v$,
  \begin{align*}
    \left\| v \right\|_\infty \leq \hat{L} \left( \left\| v \right\|^2 + \left\| \fd_x v \right\|^2 \right)^\frac{1}{2}
  \end{align*}
  holds, where $\hat{L} = \sqrt{2} \max \left\{ \sqrt{L}, 1/\sqrt{L} \right\}$.
\end{lemma}


\subsection{Schemes}

In order to consider energy-conservative schemes for the Zakharov equations, we assume that $\Ed{m}{k},\Nd{m}{k}$ and $\Vd{m}{k}$ are numerical solutions of such a scheme, $(\tEd{m}{k}, \tNd{m}{k}, \tVd{m}{k})   \coloneqq (E(m \Dt, k \Dx), N(m \Dt, k \Dx), V(m \Dt, k \Dx))$ are the solutions of \eqref{eq_Zak} and \eqref{eq_Vdef}, or
\begin{align}\label{eq_Zak2}
  \begin{cases}
    \mathrm{i} E_t + E_{xx} = NE, \\
    N_t = V_{xx}, \\
    V_t = N + \left\vert E \right\vert^2,\, V(0,0) = 0,
  \end{cases} \hspace{10mm} \left( t > 0,\, x\in(0,L) \right)
\end{align}
and $\EEd{m}{k} \coloneqq \Ed{m}{k} - \tEd{m}{k}, \ENd{m}{k} \coloneqq \Nd{m}{k} - \tNd{m}{k}$ and $\EVd{m}{k} \coloneqq \Vd{m}{k} - \tVd{m}{k}$ are the error terms.\\

In the following, we present two schemes that we address in this study.
One is the Glassey scheme (sometimes abbreviated as (G)), which has been proposed in the existing study \cite{Glassey}, and the other is the DVDM scheme (sometimes abbreviated as (D)), which is the main target of this study.
The reason why we pick up (G) among several existing schemes is that (G) is similar in form to (D) and (G) should work as a good example to evaluate the performance of (D).


\subsubsection{Glassey's scheme \rm(G)}

In \cite{Glassey}, Glassey proposed an energy-conservative scheme (G) for the Zakharov equations, which is defined as
\begin{align}\label{sch_Glassey}
  \begin{cases}
    \mathrm{i} \fd_t \Ed{m}{k} = -\cdd_x \fa_t \Ed{m}{k} + \left( \fa_t \Nd{m}{k} \right) \left( \fa_t \Ed{m}{k} \right), \\
    \fd_t \bd_t \Nd{m}{k} = \cdd_x \ca_t \Nd{m}{k} + \cdd_x \left( \left\vert \Ed{m}{k} \right\vert^2\right)
  \end{cases}
\end{align}
with the initial conditions
\begin{align}\label{eq_ini_Glassey}
  \Ed{0}{k} = E^0(k\Dx),\quad \Nd{0}{k} = N^0(k\Dx). \qquad (k=1,\ldots,K)
\end{align}
In addition, $\Nd{1}{}$ must also be set as an initial condition, which will be discussed in detail in Section 4.1.1. For analysis, $\Vd{m}{k}$ is defined by
\begin{align*}
  \fd_t \Nd{m}{k} = \cdd_x \Vd{m}{k}.
\end{align*}

For this scheme, second-order convergence is shown under the assumption $\Dt=\Dx$ in \cite{Chang}.
In addition, it is shown that the scheme is unconditionally solvable because it is a linearly-implicit scheme.
This means the following: if $\Ed{m}{}, \Nd{m}{}$ and $\Nd{m-1}{}$ are known, then $\Nd{m+1}{}$ is obtained by transforming the second formula in \eqref{sch_Glassey} as follows:
\begin{align}\label{sch_Glassey1}
  &\fd_t \bd_t \Nd{m}{k} = \cdd_x \ca_t \Nd{m}{k} + \cdd_x \left( \left\vert \Ed{m}{k} \right\vert^2\right) \nonumber\\
  &\Leftrightarrow \hspace{1mm} \left( I -  \frac{(\Dt)^2}{2} D_x^{\left\langle 2 \right\rangle} \right) \left( \Nd{m+1}{} + \Nd{m-1}{} + 2 \left\vert \Ed{m}{} \right\vert^2 \right) = 2 \Nd{m}{} + 2 \left\vert \Ed{m}{} \right\vert^2 \nonumber\\
  \begin{split}
    &\Leftrightarrow \hspace{2mm}
    \Nd{m+1}{} = - \Nd{m-1}{} - 2 \left\vert \Ed{m}{} \right\vert^2\\
    &\hspace{30mm} + 2 \left( I -  \frac{(\Dt)^2}{2} D_x^{\left\langle 2 \right\rangle} \right)^{-1} \left( \Nd{m}{} + \left\vert \Ed{m}{} \right\vert^2 \right),
  \end{split}
\end{align}
where we introduced new notations: $\left\vert \Ed{m}{} \right\vert \coloneqq \left( \left\vert \Ed{m}{k} \right\vert \right)_k^\top$ and the representation matrix of $\cdd_x$ is denoted by $D_x^{\left\langle 2 \right\rangle}$.
Note that $\Nd{1}{}$ is defined by other method as the initial condition.
On the other hand, if $\Ed{m}{}, \Nd{m}{}$ and $\Nd{m+1}{}$ are known, then $\Ed{m+1}{}$ is obtained by transforming the first formula in \eqref{sch_Glassey} as follows:
\begin{align}\label{sch_Glassey2}
  & \mathrm{i} \fd_t \Ed{m}{k} = -\cdd_x \fa_t \Ed{m}{k} + \left( \fa_t \Nd{m}{k} \right) \left( \fa_t \Ed{m}{k} \right) \nonumber\\
  &\Leftrightarrow \hspace{2mm} \left[ 2 \mathrm{i} I + \Dt \left\{ D_x^{\left\langle 2 \right\rangle} - \mathrm{diag} \left( \fa_t \Nd{m}{} \right) \right\} \right] \left( \Ed{m+1}{} + \Ed{m}{} \right) = 4 \mathrm{i} \Ed{m}{} \nonumber\\
  \begin{split}
    &\Leftrightarrow \hspace{2mm} \Ed{m+1}{} = -\Ed{m}{}\\
    &\hspace{20mm} + 4 \mathrm{i} \left[ 2 \mathrm{i} I + \Dt \left\{ D_x^{\left\langle 2 \right\rangle} - \mathrm{diag} \left( \fa_t \Nd{m}{} \right) \right\} \right]^{-1} \Ed{m}{}.
  \end{split}
\end{align}
Here, the existence of inverse matrices appearing in \eqref{sch_Glassey1} and \eqref{sch_Glassey2} is shown in the discussion in \cite{Glassey}.
Repeating these, Glassey's scheme can be solved by simply considering the linear equations.
In other words, it is fast to solve.

On the other hand, invariants of the scheme (G) are $\mathcal{N}^{(m)} = \left\| \Ed{m}{} \right\|^2$ and
\begin{align*}
  \begin{split}
    \mathcal{E}^{(m+1)}_{\mathrm{G}} &\coloneqq \left\| \fd_x \Ed{m+1}{} \right\|^2 + \frac{1}{2} \left( \fa_t \left\| \Nd{m}{} \right\|^2 + \left\| \fd_x \Vd{m}{} \right\|^2 \right)\\
    &\hspace{40mm} + \left\langle \fa_t \Nd{m}{}, \left\vert \Ed{m+1}{} \right\vert^2 \right\rangle,
  \end{split}
\end{align*}
which 
is not consistent with respect to time.
Therefore, it should be examined if such an inconsistency would affect the actual performance of the scheme.


\subsubsection{DVDM scheme \rm(D)}

By applying the discrete variational derivative method \cite{DVDM} to the Zakharov equations, we obtain
\begin{align}\label{sch_DVDM}
  \begin{cases}
    \mathrm{i} \fd_t \Ed{m}{k} = - \cdd_x \fa_t \Ed{m}{k} + \left( \fa_t \Nd{m}{k} \right) \left( \fa_t \Ed{m}{k} \right), \\
    \fd_t \Nd{m}{k} = \cdd_x \fa_t \Vd{m}{k},\\
    \fd_t \Vd{m}{k} = \fa_t \Nd{m}{k} + \fa_t \left\vert \Ed{m}{k} \right\vert^2.
  \end{cases}
\end{align}
Note that while the initial values $\Ed{0}{}$ and $\Nd{0}{}$ are determined by \eqref{eq_ini_Glassey} using $E^0$ and $N^0$, the initial value $\Vd{0}{}$ is not always strictly calculated. In the following analysis, for $V$ determined by \eqref{eq_Vdef}, we only assume that
\begin{align}\label{eq_Vassumption}
  \left\vert \Vd{0}{k} - V(0,k\Dx) \right\vert \leq C_V(\Dx)^2
\end{align}
holds for a certain constant $C_V > 0$ not depending on $\Dx$. This is a loose assumption since it is easy to perform more accurate approximate calculations. Depending on the situation, we may use the following form with $\Vd{m}{k}$ removed:
\begin{align}\label{sch_DVDM2}
  \begin{cases}
    \mathrm{i} \fd_t \Ed{m}{k} &= - \cdd_x \fa_t \Ed{m}{k} + \left( \fa_t \Nd{m}{k} \right) \left( \fa_t \Ed{m}{k} \right), \\
    \fd_t \bd_t \Nd{m}{k} &= \fa_t \ba_t \cdd_x \Nd{m}{k} + \fa_t \ba_t \cdd_x \left\vert \Ed{m}{k} \right\vert^2,
  \end{cases}
\end{align}
which is very similar to Glassey's scheme in \eqref{sch_Glassey}.

As mentioned in Introduction, this scheme was introduced in the textbook~\cite{DVDM} just to show an example for applying the DVDM to various equations, and no theoretical nor numerical investigation was done so far.

On the other hand, invariants of the scheme (D) are $\mathcal{N}^{(m)} = \left\| \Ed{m}{} \right\|^2$ and
\begin{align*}
  \mathcal{E}^{(m)}_{\mathrm{D}} \coloneqq \left\| \fd_x \Ed{m}{} \right\|^2 + \frac{1}{2} \left( \left\| \Nd{m}{} \right\|^2 + \left\| \fd_x \Vd{m}{} \right\|^2 \right) + \left\langle \Nd{m}{}, \left\vert \Ed{m}{} \right\vert^2 \right\rangle,
\end{align*}
which is consistent with respect to time.
Thus, we can expect that the solutions exhibit better behaviors compared to those in inconsistent schemes.
The scheme is, however, fully-implicit and demands much computational cost.
Thus our focus should go to how the expected good behaviors compensate the cost.


\section{Mathematical analysis}

We first perform a mathematical analysis of the DVDM scheme.
It consists of two parts: solvability and convergence.

Let us start with a lemma that is very important for the proof of the solvability and convergence.


\begin{lemma}\label{lem_sup}
  We assume that $\left( \Ed{0}{}, \Nd{0}{}, \Vd{0}{}\right), \cdots, \left( \Ed{m}{}, \Nd{m}{}, \Vd{m}{}\right)$ are solutions of the DVDM scheme in \eqref{sch_DVDM} and $E,N$ and $V$ are solutions of the Zakharov equations \eqref{eq_Zak} and \eqref{eq_Vdef} satisfying $E(t,\cdot), V(t,\cdot) \in C^1(0,L)$. Then, for sufficiently small $\Dx$,
  \begin{align*}
    \left\| \Ed{m}{} \right\|_2,
    \, \left\| \fd_x \Ed{m}{} \right\|_2,
    \, \left\| \Ed{m}{} \right\|_\infty,
    \, \left\| \Nd{m}{} \right\|_2,\, \left\| \fd_x \Vd{m}{} \right\|_2 \leq C_\infty
  \end{align*}
  holds for a certain constant $C_\infty > 0$ not depending on $\Dt, \Dx$ and $m$.
\end{lemma}


\begin{proof}
  From $\fd_t \mathcal{N}^{(m)} = \fd_t \mathcal{E}^{(m)}_{\mathrm{D}} = 0$, we obtain $\left\| \Ed{m}{} \right\| = \left\| \Ed{0}{} \right\|$ and $\mathcal{E}^{(m)}_{\mathrm{D}} = \mathcal{E}^{(0)}_{\mathrm{D}}$ for all $m$. Here, for the former, we see that
  \begin{equation*}
    \left\| \Ed{m}{} \right\| = \left\| \Ed{0}{} \right\| \leq \sqrt{L} \left\| \Ed{0}{} \right\|_\infty \leq \sqrt{L} \sup_{x\in[0,L]} \left\vert E(0,x) \right\vert \eqqcolon c_1
  \end{equation*}
  which does not depend on $\Dt,\Dx$ and $m$.
  Since
  \begin{align}\label{eq_fdx}
    \left\vert \fd_x E(t,k\Dx) \right\vert \leq \left\vert E_x(t,(k+\varepsilon)\Dx)\right\vert \leq \sup_{x\in[0,L]} \left\vert E_x(t,x) \right\vert
  \end{align}
  holds for all $t$ and some $\varepsilon \in (0,1)$ from the mean value theorem, we obtain
  \begin{align*}
    \mathcal{E}^{(0)}_{\mathrm{D}} &= \left\| \fd_x \tEd{0}{} \right\|^2 + \frac{1}{2} \left( \left\| \tNd{0}{} \right\|^2 + \left\| \fd_x \tVd{0}{} \right\|^2 \right) + \left\langle \tNd{0}{}, \left\vert \tEd{0}{} \right\vert^2 \right\rangle \leq c_2,
  \end{align*}
  for some $c_2$ not depending on $\Dt,\Dx$ and $m$.
  Thus, from $\mathcal{E}^{(m)}_{\mathrm{D}} \leq c_2$, we see that
  \begin{align*}
    \left\| \fd_x \Ed{m}{} \right\|^2 + \frac{1}{2} \left( \left\| \Nd{m}{} \right\|^2 + \left\| \fd_x \Vd{m}{} \right\|^2 \right) &\leq c_2 - \left\langle \frac{1}{2} \Nd{m}{}, 2\left\vert \Ed{m}{} \right\vert^2 \right\rangle\\
    &\leq c_2 + \frac{1}{4} \left\| \Nd{m}{} \right\|^2 + \left\| \Ed{m}{} \right\|_4^4,
  \end{align*}
  which implies
  \begin{align*}
    \left\| \fd_x \Ed{m}{} \right\|^2 + \frac{1}{4} \left\| \Nd{m}{} \right\|^2 + \frac{1}{2} \left\| \fd_x \Vd{m}{} \right\|^2 \leq c_2 + \frac{1}{2} \left\| \Ed{m}{} \right\|_4^4.
  \end{align*}

  By using Lemma 3.1 in \cite{Sun}:
  \begin{align*}
    \left\| v \right\|_p &\leq C' \left( \left\| \fd_x v \right\|^{\frac{1}{2}-\frac{1}{p}} \left\| v\right\|^{\frac{1}{2}+\frac{1}{p}} + \left\| v \right\|\right)
  \end{align*}
  for some $C'$ not depending on $p, \Dt$ and $\Dx$, and using the facts
  \begin{align*}
    (a+b)^4 \leq 8(a^4+b^4),
    \quad 2ab \leq a^2+b^2, \qquad (\forall a, b \in \mathbb{R})
  \end{align*}
  we obtain
  \begin{align*}
    \left\| \Ed{m}{} \right\|_4^4 &\leq C'^4 \left( \left\| \fd_x \Ed{m}{} \right\|^{\frac{1}{4}} \left\| \Ed{m}{} \right\|^{\frac{3}{4}} + \left\| \Ed{m}{} \right\| \right)^4\\
    &\leq 8C'^4 \left( \left\| \fd_x \Ed{m}{} \right\| \left\| \Ed{m}{} \right\|^3 + \left\| \Ed{m}{} \right\|^4 \right)\\
    &= 2 \left\| \fd_x \Ed{m}{} \right\| \cdot 4C'^4 \left\| \Ed{m}{} \right\|^3 + 8C'^4 \left\| \Ed{m}{} \right\|^4\\
    &\leq \left\| \fd_x \Ed{m}{} \right\|^2 + 16C'^8 \left\| \Ed{m}{} \right\|^6 +8C'^4 \left\| \Ed{m}{} \right\|^4,
  \end{align*}
  which implies
  \begin{align*}
    \left\| \fd_x \Ed{m}{} \right\|^2 + \frac{1}{4} \left\| \Nd{m}{} \right\|^2 + \frac{1}{2} \left\| \fd_x \Vd{m}{} \right\|^2 \leq c_3 +  \frac{1}{2} \left\| \fd_x \Ed{m}{} \right\|^2.
  \end{align*}
  for some $c_3$ not depending on $\Dt, \Dx$ and $m$. This means
  \begin{align*}
    \left\| \fd_x \Ed{m}{} \right\|^2 + \frac{1}{2} \left\| \Nd{m}{} \right\|^2 + \left\| \fd_x \Vd{m}{} \right\|^2 \leq 2c_3.
  \end{align*}
  From Lemma~\ref{lem_Sobolev}, we obtain the conclusion.
\end{proof}

The solvability is stated as follows:

\begin{theorem}[Solvability]\label{thm_solvability}
  Let us assume that $E,N$ and $V$ are solutions of the Zakharov equations \eqref{eq_Zak} and \eqref{eq_Vdef} satisfying $E(t,\cdot), V(t,\cdot) \in C^1(0,L)$, $p > 3$, and $r \geq C_\infty$ is a certain constant, where $C_\infty$ is obtained in Lemma~\ref{lem_sup}.
  If $\Dt$ and $\Dx$ are sufficiently small and satisfy $\Dt < \min\left\{\Dx, \varepsilon_1(p, r), \varepsilon_2(p, r)\right\}$ with
  \begin{align*}
    \varepsilon_1(p, r) &\coloneqq \frac{2(p-3)\Dx}{ 2 + 2p(p\hat{L} + 1)r + p(p+1)r (\Dx)^{1/2} + (2p^2\hat{L} + p + 1)r \Dx },\\
    \varepsilon_2(p, r) &\coloneqq \frac{\Dx}{r \left( 2p\hat{L} + 1 + (p+1/2) (\Dx)^{1/2} + ( 2p\hat{L} + 1/2 ) \Dx \right) },
  \end{align*}
  then, the DVDM scheme \eqref{sch_DVDM} has a unique solution $\left( \Ed{m}{}, \Nd{m}{}, \Vd{m}{}\right)$ for $m = 1,2, \ldots, M$.
\end{theorem}


\begin{proof}
  It suffices to show that, if $\left( \Ed{0}{}, \Nd{0}{}, \Vd{0}{}\right), \cdots, \left( \Ed{m}{}, \Nd{m}{}, \Vd{m}{}\right)$ are solutions of the DVDM scheme in \eqref{sch_DVDM}, then $\Ed{m+1}{}$ and $\Nd{m+\frac{1}{2}}{} \coloneqq \fa_t \Nd{m}{}$ uniquely exist, for all $m = 0,1,\ldots,M-1$.
  In this time, the unique existence of $\left( \Ed{m+1}{}, \Nd{m+1}{}, \Vd{m+1}{} \right)$ follows from $\Nd{m+1}{} = 2\Nd{m+\frac{1}{2}}{} - \Nd{m}{}$ and
  \begin{align*}
    \Vd{m+1}{k} = \Vd{m}{k} + \Dt \Nd{m+\frac{1}{2}}{k} + \dfrac{\Dt}{2} \left( \left\vert \Ed{m+1}{k} \right\vert^2 + \left\vert \Ed{m}{k} \right\vert^2 \right),
  \end{align*}
  which is derived from \eqref{sch_DVDM}.

  In order to show the unique existence of $\left( \Ed{m+1}{}, \Nd{m+\frac{1}{2}}{} \right)$ by applying a fixed point theorem, we transform the definition of the DVDM scheme \eqref{sch_DVDM} as
  \begin{align*}
    &\begin{cases}
      \mathrm{i} \fd_t \Ed{m}{k} = - \cdd_x \fa_t \Ed{m}{k} + \left( \fa_t \Nd{m}{k} \right) \left( \fa_t \Ed{m}{k} \right), \\
      \fd_t \Nd{m}{k} = \cdd_x \fa_t \Vd{m}{k},\\
      \fd_t \Vd{m}{k} = \fa_t \Nd{m}{k} + \fa_t \left\vert \Ed{m}{k} \right\vert^2
    \end{cases}\\
    &\Leftrightarrow
    \begin{cases}
      \Ed{m+1}{} = \left( \mathrm{i} I + \dfrac{\Dt}{2} D_x^{\left\langle 2 \right\rangle} \right)^{-1}  \left( \mathrm{i} I - \dfrac{\Dt}{2} D_x^{\left\langle 2 \right\rangle} \right) \Ed{m}{}\\
      \hspace{19mm} +  \dfrac{\Dt}{2} \left( \mathrm{i} I + \dfrac{\Dt}{2} D_x^{\left\langle 2 \right\rangle} \right)^{-1}  \Nd{m+\frac{1}{2}}{} \left( \Ed{m+1}{} + \Ed{m}{} \right), \\
      \Nd{m+\frac{1}{2}}{} = \left( I - \dfrac{(\Dt)^2}{4} D_x^{\left\langle 2 \right\rangle} \right)^{-1} \biggl\{ \Nd{m}{} + \dfrac{\Dt}{2} \bd_x \fd_x \Vd{m}{}\\
      \hspace{50mm}+ \dfrac{(\Dt)^2}{8} D_x^{\left\langle 2 \right\rangle} \left( \left\vert \Ed{m+1}{} \right\vert^2 + \left\vert \Ed{m}{} \right\vert^2 \right) \biggr\}.
    \end{cases}
  \end{align*}
  Then, we define the following maps:
  \begin{align*}
    \begin{cases}
      \phi^{(m)}_E (E,N) \coloneqq \left( \mathrm{i} I + \dfrac{\Dt}{2} D_x^{\left\langle 2 \right\rangle} \right)^{-1}  \left( \mathrm{i} I - \dfrac{\Dt}{2} D_x^{\left\langle 2 \right\rangle} \right) \Ed{m}{}\\
      \hspace{23mm} +  \dfrac{\Dt}{2} \left( \mathrm{i} I + \dfrac{\Dt}{2} D_x^{\left\langle 2 \right\rangle} \right)^{-1} N \left( E + \Ed{m}{} \right), \\
      \phi^{(m)}_N (E,N) \coloneqq \left( I - \dfrac{(\Dt)^2}{4} D_x^{\left\langle 2 \right\rangle} \right)^{-1} \Biggl\{ \Nd{m}{} + \dfrac{\Dt}{2} \bd_x \fd_x \Vd{m}{}\\
      \hspace{52mm} + \dfrac{(\Dt)^2}{8} D_x^{\left\langle 2 \right\rangle} \left( \left\vert E \right\vert^2 + \left\vert \Ed{m}{} \right\vert^2 \right) \Biggr\}
    \end{cases}
  \end{align*}
  and $\phi^{(m)} (E,N) \coloneqq \left( \phi^{(m)}_E (E,N), \phi^{(m)}_N (E,N) \right)$, which satisfy that
  \begin{align*}
    \phi^{(m)} (E^\ast,N^\ast) = (E^\ast,N^\ast) \quad \Leftrightarrow \quad E^\ast = \Ed{m+1}{}, N^\ast = \Nd{m+\frac{1}{2}}{}.
  \end{align*}

  Therefore, it suffices to apply the fixed point theorem to the function $\phi^{(m)}$, and for that purpose, we show that $\phi^{(m)}$ is a contraction mapping on some sphere $B(pr) \coloneqq \left\{ (E, N) \mid \left\| (E, N) \right\| \leq pr \right\}$, where $\left\| (E, N) \right\| \coloneqq \left\| E \right\|_2 + \left\| \fd_x E \right\|_2 + \left\| N \right\|_2$ is a norm. Note that, if $(E,N) \in B(pr)$, then $\left\| E \right\|_\infty \leq p\hat{L}r$.\\

  First let us evaluate norms of some matrices appearing in the above maps.
  Let $\lambda_1, \ldots, \lambda_K$ be the eigenvalues of $D_x^{\left\langle 2 \right\rangle}$ and $\sigma_1,\ldots,\sigma_K$ be the singular values of the normal matrix $\mathrm{i} I + \Dt D_x^{\left\langle 2 \right\rangle}/2$, then $\left\{ \sigma_k \right\}_k$ is equal to $\left\{ \vert i + \lambda_k \Dt/2 \vert \right\}_k$, i.e., $\left\vert \sigma_k \right\vert \geq 1$ for all $k = 1,\ldots,K$, and accordingly $\| ( \mathrm{i} I + \Dt D_x^{\left\langle 2 \right\rangle}/2 )^{-1} \| \leq 1$ holds.
  The eigenvalues of $( \mathrm{i} I + \Dt D_x^{\left\langle 2 \right\rangle}/2 )^{-1} ( \mathrm{i} I - \Dt D_x^{\left\langle 2 \right\rangle}/2 )$ are $\left( \mathrm{i} + \lambda_k \Dt/2\right)^{-1} \left( \mathrm{i} - \lambda_k \Dt/2 \right)$ $(k = 1,\ldots,K)$,
  whose absolute values are obviously all $1$.
  This immediately implies $\| ( \mathrm{i} I + \Dt D_x^{\left\langle 2 \right\rangle}/2 )^{-1} ( \mathrm{i} I - \Dt D_x^{\left\langle 2 \right\rangle}/2 ) \|$ = 1.
  It is also easy to see that $\| ( I - (\Dt)^2 D_x^{\left\langle 2 \right\rangle}/4 )^{-1} \| \leq 1$ holds.

  Under these observations, now let us evaluate $\phi^{(m)}$.
  First, we show that, if $\Dt \leq \Dx$ and $\Dt \leq \varepsilon_1(p, r)$, then $\phi^{(m)} \left( B(pr) \right) \subseteq B(pr)$ holds.
  In fact, for $(E,N) \in B(pr)$, since
  \begin{align*}
    \left\| \phi^{(m)}_E (E,N) \right\| &\leq \left\| \left( \mathrm{i}I + \dfrac{\Dt}{2} D_x^{\left\langle 2 \right\rangle} \right)^{-1} \left( \mathrm{i}I - \dfrac{\Dt}{2} D_x^{\left\langle 2 \right\rangle} \right) \right\| \left\| \Ed{m}{} \right\|\\
    &\hspace{5mm} + \dfrac{\Dt}{2} \left\| \left( \mathrm{i}I + \dfrac{\Dt}{2} D_x^{\left\langle 2 \right\rangle} \right)^{-1} \right\| \left\| N \right\| \left\| E + \Ed{m}{} \right\|_\infty \\
    &\leq r + \frac{\Dt}{2} p(p\hat{L} + 1)r^2,\\
    \left\| \fd_x \phi^{(m)}_E (E,N) \right\| &\leq \left\| \fd_x \Ed{m}{} \right\|\\
    &\hspace{5mm} + \dfrac{\Dt}{2} \left( \left\| \fd_x N \right\| \left\| E + \Ed{m}{} \right\|_\infty + \frac{\left\| N \right\| \left\| \fd_x \left( E + \Ed{m}{} \right) \right\|}{(\Dx)^{1/2}} \right) \\
    &\leq r + \frac{\Dt}{2\Dx} \left( 2p(p\hat{L} + 1) + p(p+1) (\Dx)^{1/2} \right) r^2,\\
    \left\| \phi^{(m)}_N (E,N) \right\| &\leq \left\| \left( I - \dfrac{(\Dt)^2}{4} D_x^{\left\langle 2 \right\rangle} \right)^{-1} \right\| \Biggl\{ \left\| \Nd{m}{} \right\| + \dfrac{\Dt}{2} \left\| \bd_x \right\| \left\| \fd_x \Vd{m}{} \right\|\\
    & \hspace{5mm} + \dfrac{(\Dt)^2}{8} \left\| \bd_x \right\| \left( 2 \left\| E \right\|_\infty \left\| \fd_x E \right\| + 2 \left\| \Ed{m}{} \right\|_\infty \left\| \fd_x \Ed{m}{} \right\| \right) \Biggr\}\\
    &\leq r + \dfrac{\Dt}{2\Dx} \left( 2 + (p^2\hat{L} + 1)r \Dx \right)r ,
  \end{align*}
  we see that the sufficient condition for $\phi^{(m)} (E,N) \in B(pr)$ to be
  \begin{align*}
    \left\| \phi^{(m)} (E,N) \right\|
    &\leq 3r +
    \dfrac{r\Dt}{2\Dx} \left( 2 + 2p(p\hat{L} + 1)r + p(p+1)r \sqrt{\Dx} + (2p^2\hat{L} + p + 1)r\Dx \right)\\
    &\leq pr\\
    &\Leftrightarrow \, \Dt \leq \frac{2(p-3)\Dx}{ 2 + 2p(p\hat{L} + 1)r + p(p+1)r (\Dx)^{1/2} + (2p^2\hat{L} + p + 1)r \Dx },
  \end{align*}
  where the most right hand side is denoted as $\varepsilon_1(p, r)$.

  Next, we show that, if $\Dt \leq \Dx$ and $\Dt \leq \varepsilon_2(p, r)$, then $\phi^{(m)}$ is a contraction mapping.
  For $(E_1,N_1), (E_2,N_2) \in B(pr)$, since
  \begin{align*}
    N_1 E_1 - N_2 E_2 &= N_1 ( E_1 - E_2 ) + \left( N_1 - N_2 \right) E_2,\\
    \left\vert E_1 \right\vert^2 - \left\vert E_2 \right\vert^2 &= \mathrm{Re} \left[ \left( \bar{E_1} + \bar{E_2} \right) \left( E_1 - E_2 \right) \right],
  \end{align*}
  we see that
  \begin{align*}
    &\left\| \phi^{(m)}_E (E_1,N_1) - \phi^{(m)}_E (E_2,N_2) \right\|\\ &\leq \dfrac{\Dt}{2} \Biggl\{ \left\| N_1 \right\| \left\| E_1 - E_2 \right\|_\infty + \left\| N_1 - N_2 \right\| \left\| E_2 +\Ed{m}{} \right\|_\infty \Biggr\}\\
    &\leq \frac{\Dt}{2} \left( 2p\hat{L} + 1 \right) r \left\| (E_1,N_1) - (E_2,N_2) \right\|,\\
    &\left\| \fd_x \left( \phi^{(m)}_E (E_1,N_1) - \phi^{(m)}_E (E_2,N_2) \right) \right\|\\ &\leq \dfrac{\Dt}{2} \Biggl\{ \left\| \fd_x N_1 \right\| \left\| E_1 - E_2 \right\|_\infty + (\Dx)^{-1/2} \left\| N_1 \right\| \left\| \fd_x \left( E_1 - E_2 \right) \right\| \\
    &\hspace{0mm} + \left\| \fd_x \left( N_1 - N_2 \right) \right\| \left\| E_2 +\Ed{m}{} \right\|_\infty + (\Dx)^{-1/2} \left\| N_1 - N_2 \right\| \left\| \fd_x \left( E_2 +\Ed{m}{}  \right) \right\| \Biggr\}\\
    &\leq \frac{\Dt}{\Dx} \left[ 2p\hat{L} + 1 + (p+1/2) (\Dx)^{1/2} \right] r \left\| (E_1,N_1) - (E_2,N_2) \right\|,\\
    &\left\| \phi^{(m)}_N (E_1,N_1) - \phi^{(m)}_N (E_2,N_2) \right\|\\ &\leq \dfrac{(\Dt)^2}{8} \left\| \bd_x \right\| \left( \left\| \fd_x \left( \bar{E_1} + \bar{E_2} \right) \right\| \left\| E_1 - E_2 \right\|_\infty + \left\| \bar{E_1} + \bar{E_2} \right\|_\infty \left\| \fd_x \left( E_1 - E_2 \right) \right\| \right)\\
    &\leq \frac{\Dt}{2} \cdot p\hat{L} r \left\| (E_1,N_1) - (E_2,N_2) \right\|.
  \end{align*}
  Therefore,
  \begin{align*}
    &\, \left\| \phi^{(m)} (E_1,N_1) - \phi^{(m)} (E_2,N_2) \right\|\\
    &\leq \frac{\Dt}{\Dx} \left( 2p\hat{L} + 1 + (p+1/2) (\Dx)^{1/2} + ( 2p\hat{L} + 1/2 ) \Dx \right) r \left\| (E_1,N_1) - (E_2,N_2) \right\|\\
    &< \left\| (E_1,N_1) - (E_2,N_2) \right\|\\
    &\Leftrightarrow \hspace{2mm} \Dt < \frac{\Dx}{r \left( 2p\hat{L} + 1 + (p+1/2) (\Dx)^{1/2} + ( 2p\hat{L} + 1/2 ) \Dx \right) } \eqqcolon \varepsilon_2(p, r).
  \end{align*}

  Thus, if $\Dt < \min\{ \Dx, \varepsilon_1(p, r), \varepsilon_2(p, r) \}$, $\phi^{(m)}$ is the desired contraction mapping and has only one fixed point.
\end{proof}

Recall that the Glassey scheme in \eqref{sch_Glassey} does not need any assumption for solvability because it is a linearly-implicit scheme.
Also, as discussed below, the unnecessary restriction $\Dt = \Dx$ can be removed in the discussion of Glassey's argument of convergence estimate. This means that there is strictly no restriction in Glassey's scheme.
On the other hand, in the DVDM scheme, we need the assumption $\Dt = \mathrm{O} \left( \Dx \right)$, which comes from the solvability.
We feel it rather comes from a technical reason of the proof (generally it is quite likely in similar mathematical analyses that too severe step size restriction appears due to the evaluation of finite difference operators).
In fact, as will be shown later, the DVDM scheme practically works well even when $\Dt=\Dx$.




\begin{theorem}[Convergence]\label{thm_convergence}
  We assume that $\left( E, N, V \right)$ is a sufficiently smooth solution of \eqref{eq_Zak2} and  $\left( \Ed{m}{}, \Nd{m}{}, \Vd{m}{} \right)$ $(m=1, \ldots, M)$ is a numerical solution of the DVDM scheme \eqref{sch_DVDM}.
  Then, for any $m = 0, \ldots, M$, and sufficiently small $\Dt, \Dx$,
  \begin{align}\label{def:Ee}
    \mathcal{E}_e^{(m)} &\coloneqq \left\| \eEd{m}{} \right\|^2 + \left\| \fd_x \eEd{m}{} \right\|^2 + \left\| \eNd{m}{} \right\|^2 + \left\| \fd_x \eVd{m}{} \right\|^2\\
    &\leq C^2 \left( (\Dt)^2+(\Dx)^2 \right)^2 \nonumber
  \end{align}
  holds for a certain constant $C > 0$ not depending on $\Dt$ and $\Dx$.
  In particular,
  \begin{align*}
    \left\| \eEd{m}{} \right\|, \left\| \fd_x \eEd{m}{} \right\|, \left\| \eNd{m}{} \right\|, \left\| \fd_x \eVd{m}{} \right\| \leq C \left( (\Dt)^2+(\Dx)^2 \right)
  \end{align*}
  holds for a certain constant $C > 0$ not depending on $\Dt$ and $\Dx$.
\end{theorem}

The proof of this theorem is partly similar to Glassey's argument in \cite{Glassey}, but is different in the following senses.
First, since the target scheme and its associated discrete energy function are a bit different, the evaluation needs to be fine-tuned (although details are omitted for reasons of space).
Second, even in the arguments borrowed from \cite{Glassey}, we reformulated the original arguments so that the strategy of the whole argument is more visible (see the explanation below Lemma~\ref{lem_gronwall}).
Third, we show that the assumption $\Dt=\Dx$ originally required in the previous study can be removed.
This rather benefits Glassey's scheme (the same improvement can be done for Glassey's argument), where no other step size restriction appears.


Let us first prepare some notation and essential lemmas.
We define the local truncation errors $\TaEd{m}{k}$, $\TaNd{m}{k}$ and $\TaVd{m}{k}$ by
\begin{align*}
  \mathrm{i} \fd_t \tEd{m}{k} &= - \cdd_x \fa_t \tEd{m}{k} + \left( \fa_t \tNd{m}{k} \right) \left( \fa_t \tEd{m}{k} \right)+ \TaEd{m}{k}, \\
  \fd_t \tNd{m}{k} &= \cdd_x \fa_t \tVd{m}{k} + \TaNd{m}{k},\\
  \fd_t \tVd{m}{k} &= \fa_t \tNd{m}{k} + \fa_t \left\vert \tEd{m}{k} \right\vert^2 + \TaVd{m}{k}.
\end{align*}
From this, we obtain
\begin{align*}
  \mathrm{i} \fd_t \eEd{m}{} &+ \cdd_x \fa_t \eEd{m}{} = \left( \fa_t \Nd{m}{} \right) \left(\fa_t \eEd{m}{}\right) + \left( \fa_t \eNd{m}{} \right) \left(\fa_t \tEd{m}{}\right) - \taEd{m}{}, \\
  \fd_t \eNd{m}{} &= \cdd_x \fa_t \eVd{m}{} - \taNd{m}{},\\
  \fd_t \eVd{m}{} &= \fa_t \eNd{m}{} + \fa_t \left( \left\vert \Ed{m}{} \right\vert^2 - \left\vert\tEd{m}{} \right\vert^2 \right) - \taVd{m}{},
\end{align*}
where $(v \ast w)_k \coloneqq v_kw_k$ is the Hadamard product, which is abbreviated as $v\ast w = vw$, $v\ast v=v^2$ or $z\ast\bar{z} = \vert z \vert^2$ below to save space (mind that we will frequently use this abbreviation).


The next lemma gives essential local truncation error estimate.

\begin{lemma}\label{lem_truncation}
  If the solution of \eqref{eq_Zak2} is sufficiently smooth, for $m=0,\ldots,M-1$
  \begin{align*}
    \left\| \taEd{m}{} \right\|, \left\| \fd_x \taEd{m}{} \right\|, \left\| \taNd{m}{} \right\|, \left\| \fd_x \taVd{m}{} \right\| \leq c_0 \left( (\Dt)^2 + (\Dx)^2 \right)
  \end{align*}
  holds for a certain constant $c_0 > 0$ not depending on $\Dt$ and $\Dx$.
\end{lemma}


\begin{proof}
  This Lemma follows from the Taylor expansion.
\end{proof}


The next lemma is of the Gronwall type, which is also essential for bounding global error estimates based on local ones.

\begin{lemma}\label{lem_gronwall}
  Suppose that $\vd{m}{}$ is a sequence with $\vd{m}{} \geq 0$ for all $m$. Suppose also that certain constants $c, d > 0$ exist such that
  \begin{align}\label{eq_gen}
    \fd_t \vd{m}{} \leq c \left( \vd{m+1}{} + \vd{m}{} \right) + d
  \end{align}
  for all $m \leq M-1$.
  Then, for sufficiently small $\Dt$ and $\Dx$,
  \begin{align*}
    \vd{m}{} \leq 2(dT + \vd{0}{}) \mathrm{exp}(4cT)
  \end{align*}
  holds for all $m = 0,\ldots,M$,.
\end{lemma}


\begin{proof}
  This lemma follows from the discrete Gronwall Lemma \cite{gronwall}.
\end{proof}


Now let us briefly state our strategy for proving Theorem~\ref{thm_convergence}, which inevitably gets complicated.
In continuous partial differential equation theories, often invariants are useful. 
Similarly, in structure-preserving schemes, discrete invariants often give a starting point of convergence estimates. In the present case, this is the energy of the DVDM scheme \eqref{sch_DVDM}:
\begin{align*}
  \mathcal{E}^{(m)}_{\mathrm{D}} = \left\| \fd_x \Ed{m}{} \right\|^2 + \frac{1}{2} \left( \left \|\Nd{m}{} \right\|^2 + \left\| \fd_x \Vd{m}{} \right\|^2 \right) + \left\langle \Nd{m}{}, \left\vert \Ed{m}{} \right\vert^2 \right\rangle,
\end{align*}
from which we naturally expect to evaluate
\begin{align}
  \left\| \fd_x \eEd{m}{} \right\|^2 + \frac{1}{2} \left( \left\| \eNd{m}{} \right\|^2 + \left\|\fd_x \eVd{m}{} \right\|^2 \right) + \left\langle \eNd{m}{}, \left\vert \eEd{m}{} \right\vert^2 \right \rangle.
\end{align}
The last inner product can be negative, but it can be bounded from above as
\begin{align*}
  \left\vert \left\langle \eNd{m}{}, \left\vert \eEd{m}{} \right\vert^2 \right \rangle \right\vert \leq \left\| \eEd{m}{} \right\|_\infty \left\langle \left\vert \eNd{m}{} \right\vert, \left\vert \eEd{m}{} \right\vert \right \rangle \leq \dfrac{C_\infty}{2} \left( \left\| \eNd{m}{} \right\|^2 + \left\| \eEd{m}{} \right\|^2 \right),
\end{align*}
which suggests that we instead consider
\begin{align}\label{eq_errorenergy}
  \left\| \fd_x \eEd{m}{} \right\|^2
  + \frac{1}{2} \left( \left\| \eNd{m}{} \right\|^2 + \left\|\fd_x \eVd{m}{} \right\|^2 \right)
  + \dfrac{C_\infty}{2} \left( \left\| \eNd{m}{} \right\|^2 + \left\| \eEd{m}{} \right\|^2 \right),
\end{align}
where $C_\infty$ is the constant defined in Lemma~\ref{lem_sup}.
This, along with Lemma~\ref{lem_gronwall} we eventually use, suggests that we define a local error estimate by, for example,
\begin{align*}
  \mathcal{E}_e^{(m)} \coloneqq \left\| \eEd{m}{} \right\|^2 + \left\| \fd_x \eEd{m}{} \right\|^2 + \left( \left\| \eNd{m}{} \right\|^2 + \left\| \fd_x \eVd{m}{} \right\|^2 \right).
\end{align*}
It was actually employed in~\eqref{def:Ee}.
Unfortunately, however, this does not immediately give rise to an estimate of the form
\begin{align}\label{eq_eqfde}
  \fd_t \mathcal{E}_e^{(m)} \leq C \left( \mathcal{E}_e^{(m+1)} + \mathcal{E}_e^{(m)} \right) + C_0^2 \left( (\Dt)^2+(\Dx)^2 \right)^2
\end{align}
for some constants $C_0, C > 0$ (not depending on $\Dt$ and $\Dx$).
Thus we need to take several additional steps.

The first additional step to take is to relax \eqref{eq_eqfde} to
\begin{align}
  \fd_t \left( \mathcal{E}_e^{(m)} + A^{(m)} \right) \leq C \left( \mathcal{E}_e^{(m+1)} + \mathcal{E}_e^{(m)} \right) + C_0^2 \left( (\Dt)^2+(\Dx)^2 \right)^2,
\end{align}
where $A^{(m)} \approx \left\langle \Nd{m}{}, \vert\Ed{m}{}\vert^2\right\rangle$ is an adjustment term required to establish such an inequality, which will be identified in Lemma~\ref{lem:Ee:2}.
This inequality is, unfortunately, still not what we seek as is, since $A^{(m)}$ is not necessarily non-negative, and we are still not be able to apply Lemma~\ref{lem_gronwall}.

Thus we take a second additional step, where we consider an alternative to $\mathcal{E}_e^{(m)}$, given by
\begin{align*}
  \hat{\mathcal{E}}_e^{(m)} &\coloneqq \gamma_1 \left\| \eEd{m}{} \right\|^2 + \left\| \fd_x \eEd{m}{} \right\|^2 + \gamma_2 \left( \left\| \eNd{m}{} \right\|^2 + \left\| \fd_x \eVd{m}{} \right\|^2 \right) + A^{(m)}.
\end{align*}
The only difference is that parameters $\gamma_1, \gamma_2 > 0$ are newly included, which is useful to finalize our argument.
Below we will give in Lemma~\ref{lem:tH:bound} a bound on $\vert A^{(m)}\vert$, and show that if we take sufficiently large $\gamma_1, \gamma_2 >0$, $\mathcal{E}_e^{(m)}$ becomes non-negative (Lemma~\ref{lem:hEe:bound}).
This enables us to (finally) apply the Gronwall-type Lemma~\ref{lem_gronwall} to obtain a global estimate for $\hat{\mathcal{E}}_e^{(m)}$.
Furthermore, in Lemma~\ref{lem:hEe:bound}, we will also show that under the same assumptions $\hat{\mathcal{E}}_e^{(m)} \geq \mathcal{E}_e^{(m)}$ holds.
This completes the proof of Theorem~\ref{thm_convergence}.

\vspace{4mm}

Let us carry out the strategy above.
Note that Lemma~\ref{lem_sup} gives useful bounds for some quantities.


Now we evaluate each of the three terms of $\fd_t \mathcal{E}_e^{(m)}$ defined in~\eqref{def:Ee} in the following three lemmas.
The first, third and fourth terms can be evaluated along the line of the discussion in Glassey~\cite{Glassey}, so the proof of this evaluation (Lemma~\ref{lem:Ee:1,3}) is omitted.

\begin{lemma}\label{lem:Ee:1,3}
  If $\Dt$ and $\Dx$ are sufficiently small, then, for $m = 0,\ldots,M-1$,
  \begin{align*}
    \fd_t \left\| \eEd{m}{} \right\|^2 &\leq C_1 \left( \mathcal{E}_e^{(m+1)} + \mathcal{E}_e^{(m)} \right) + (c_0)^2 \left( (\Dx)^2+(\Dt)^2 \right)^2,\\
    \fd_t \left( \left\| \eNd{m}{} \right\|^2 + \left\| \fd_x \eVd{m}{} \right\|^2 \right) &\leq C_2 \left( \mathcal{E}_e^{(m+1)} + \mathcal{E}_e^{(m)} \right) + 2(c_0)^2 \left( (\Dx)^2+(\Dt)^2 \right)^2
  \end{align*}
  holds for $c_0$ defined by Lemma \ref{lem_truncation} and certain constants $C_1, C_2 > 0$ not depending on $\Dt$ and $\Dx$.
\end{lemma}

The evaluation of the second term of $\fd_t \mathcal{E}_e^{(m)}$ becomes a bit cumbersome, where we need to add an adjustment term ($A^{(m)}$ below).
Note also that the proof of the following Lemma is improved from Glassey's original argument in~\cite{Glassey}.


\begin{lemma}\label{lem:Ee:2}
  Suppose $\Dt, \Dx$ are sufficiently small.
  Then, for $m = 0,\ldots,M-1$,
  \begin{align*}
    \fd_t \left\| \fd_x \eEd{m}{} \right\|^2 \leq - \fd_t A^{(m)} + C_3 \left( \mathcal{E}_e^{(m+1)} + \mathcal{E}_e^{(m)} \right) + 4(c_0)^2 \left( (\Dx)^2+(\Dt)^2 \right)^2,
  \end{align*}
  where
  \begin{align*}
    A^{(m)} \coloneqq \left\langle \eNd{m}{} + \tNd{m}{}, \left\vert \eEd{m}{} \right\vert^2 \right\rangle + 2\mathrm{Re} \left\langle \eEd{m}{}, \tEd{m}{} \eNd{m}{} \right\rangle,
  \end{align*}
  holds for  $c_0$ defined by Lemma~\ref{lem_truncation} and a certain constant $C_3 > 0$ not depending on $\Dt$ and $\Dx$.
\end{lemma}


\begin{proof}
  By definition, we see that
  \begin{align}\label{eq_eval_ex}
    \fd_t \left\| \fd_x \eEd{m}{} \right\|^2 &= 2 \mathrm{Re} \left\langle \fd_t \fd_x \eEd{m}{}, \fa_t \fd_x \eEd{m}{} \right\rangle \nonumber\\
    &= -2 \mathrm{Re} \left[ \mathrm{i} \left\| \fd_t \eEd{m}{} \right\|^2 \right]-2 \mathrm{Re} \left\langle \fd_t \eEd{m}{}, \fa_t \cdd_x \eEd{m}{} \right\rangle \nonumber\\
    &= -2 \mathrm{Re} \left\langle \fd_t \eEd{m}{}, \mathrm{i} \fd_t \eEd{m}{} + \fa_t \cdd_x \eEd{m}{} \right\rangle \nonumber\\
    \begin{split}
      &= -2\mathrm{Re} \left\langle \fd_t \eEd{m}{}, \left( \fa_t \Nd{m}{} \right) \left( \fa_t \eEd{m}{} \right) \right\rangle\\
      &\hspace{0mm} - 2\mathrm{Re} \left\langle \fd_t \eEd{m}{}, \left( \fa_t \tEd{m}{} \right) \left( \fa_t \eNd{m}{} \right) \right\rangle + 2 \mathrm{Re} \left\langle \fd_t \eEd{m}{}, \taEd{m}{} \right\rangle.
    \end{split}
  \end{align}

  Here, the last term of \eqref{eq_eval_ex} is evaluated as
  \begin{align*}
    &2\mathrm{Re} \left\langle \fd_t \eEd{m}{}, \taEd{m}{} \right\rangle\\
    &= 2\mathrm{Im}\left[ \left\langle - \fa_t \cdd_x \eEd{m}{} + \left( \fa_t \Nd{m}{} \right) \left( \fa_t \eEd{m}{} \right) + \left( \fa_t \eNd{m}{} \right) \left( \fa_t \tEd{m}{} \right), \taEd{m}{} \right\rangle \right]\\
    &\leq \left\| \fa_t \fd_x \eEd{m}{} \right\|^2 + \left\| \fd_x \taEd{m}{} \right\|^2 + \left\| \left( \fa_t \Nd{m}{} \right) \left( \fa_t \eEd{m}{} \right) \right\|^2 + \left\| \taEd{m}{} \right\|^2\\
    &\hspace{50mm} + \left\| \left( \fa_t \tEd{m}{} \right) \left( \fa_t \eNd{m}{} \right) \right\|^2 + \left\| \taEd{m}{} \right\|^2 \\
    &\leq \left( 1 + (C_\infty)^2 + \left\| \fa_t \tNd{m}{} \right\|_\infty^2 \right) \fa_t \mathcal{E}_e^{(m)} + 3(c_0)^2 \left( (\Dt)^2 + (\Dx)^2 \right)^2.
  \end{align*}

  Therefore, it suffices to evaluate the first and second terms of \eqref{eq_eval_ex}:
  \begin{align*}
    I_1 &\coloneqq -2 \mathrm{Re} \left\langle \fd_t \eEd{m}{}, \left(\fa_t \Nd{m}{} \right) \left( \fa_t \eEd{m}{} \right) \right\rangle, \\
    I_2 &\coloneqq - 2\mathrm{Re} \left\langle \fd_t \eEd{m}{}, \left( \fa_t \tEd{m}{} \right) \left( \fa_t \eNd{m}{} \right) \right\rangle.
  \end{align*}

  Since the definition of $\eEd{m}{}$ is the same as in Glassey's study~\cite{Glassey}, we can treat $I_1$ and $I_2$ in exactly the same way. In other words, we have

  \begin{align*}
    I_1 \leq -\fd_t A_1^{(m)} + 
    C_{3,1} 
    \left( \mathcal{E}_e^{(m+1)} +\mathcal{E}_e^{(m)} \right)
  \end{align*}
  for some $C_{3,1}>0$ and $A_1^{(m)} \coloneqq \left\langle \Nd{m}{}, \left\vert \eEd{m}{} \right\vert^2 \right\rangle = \left\langle \eNd{m}{} + \tNd{m}{}, \left\vert \eEd{m}{} \right\vert^2 \right\rangle$, and 

  \begin{align*}
    I_2 \leq - \fd_t A_2^{(m)} + C_{3,2} 
    \left( \mathcal{E}_e^{(m+1)} +\mathcal{E}_e^{(m)} \right) + (c_0)^2 \left( (\Dt)^2 + (\Dx)^2 \right)^2
  \end{align*}
  for some $C_{3,2} > 0$ and $A_2^{(m)} \coloneqq 2\mathrm{Re} \left\langle \eEd{m}{}, \tEd{m}{} \eNd{m}{} \right\rangle$. Summing up the above estimates, we finally obtain
  \begin{align*}
    \fd_t \left\| \fd_x \eEd{m}{} \right\|^2
    &= I_1 + I_2 + 2 \mathrm{Re} \left\langle \fd_t \eEd{m}{}, \taEd{m}{} \right\rangle\\
    &\leq -\fd_t A^{(m)} + C_3 \left( \mathcal{E}_e^{(m+1)} +\mathcal{E}_e^{(m)} \right) + 4(c_0)^2 \left( (\Dt)^2 + (\Dx)^2 \right)^2
  \end{align*}
  for some $C_3>0$.
  Note that $A^{(m)} = A_1^{(m)} + A_2^{(m)}$.
\end{proof}

\begin{remark}
  The last term of \eqref{eq_eval_ex} is evaluated by Glassey in \cite{Glassey} as
  \begin{align*}
    \mathrm{Re} \left\langle \fd_t \eEd{m}{}, \taEd{m}{} \right\rangle &\leq \frac{1}{\Dt} \left\|\eEd{m+1}{} - \eEd{m}{} \right\| \left\| \taEd{m}{} \right\|\\
    &\leq \left( \left\|\eEd{m+1}{} \right\| + \left\| \eEd{m}{} \right\| \right) \mathrm{O}\left(\frac{(\Dx)^2 + (\Dt)^2}{\Dt}\right)\\
    &= \mathrm{O} \left( \Dt \cdot \fa_t \left( \mathcal{E}_e^{(m)} \right)^{\frac{1}{2}} \right) = \mathrm{O}\left( \mathcal{E}_e^{(m+1)} + \mathcal{E}_e^{(m)} \right)
  \end{align*}
  under the assumption $\Dt = \Dx$.
  Observe that this assumption is no longer required in the above proof.
\end{remark}


At this moment, we have
\begin{align*}
  \fd_t \left( \mathcal{E}_e^{(m)} + A^{(m)} \right) \leq C' \left( \mathcal{E}_e^{(m+1)} + \mathcal{E}_e^{(m)} \right) + 7(c_0)^2 \left( (\Dt)^2 + (\Dx)^2 \right)^2
\end{align*}
for $C' \coloneqq C_1 + C_2 + C_3$. This still does not fit Lemma \ref{lem_gronwall} due to the additional term $A^{(m)}$.
The next lemma gives a bound for this term to cancel it out.


\begin{lemma}\label{lem:tH:bound}
  For $m = 0,\ldots,M$, $\left\vert A^{(m)} \right\vert \leq \dfrac{3}{2} \left\| \eNd{m}{} \right\|^2 + C_4 \left\| \eEd{m}{} \right\|^2$ holds for a certain constant $C_4 > 0$ not depending on $\Dt$ and $\Dx$.
\end{lemma}


\begin{proof}
  By definition,
  \begin{align*}
    \left\vert A^{(m)} \right\vert &= \left\vert \left\langle \eNd{m}{} + \tNd{m}{}, \left\vert \eEd{m}{} \right\vert^2 \right\rangle + 2\mathrm{Re} \left\langle \eEd{m}{} \tEd{m}{}, \eNd{m}{} \right\rangle \right\vert\\
    &\leq \frac{3}{2} \left\| \eNd{m}{} \right\|^2 + \left( \frac{1}{2} \left\| \eEd{m}{} \right\|_\infty^2 + \left\| \tNd{m}{} \right\|_\infty^2 + \left\| \tEd{m}{} \right\|_\infty^2 \right) \left\| \eEd{m}{} \right\|^2\\
    &\leq \frac{3}{2} \left\| \eNd{m}{} \right\|^2 + C_4 \left\| \eEd{m}{} \right\|^2
  \end{align*}
  holds for some $C_4>0$.
\end{proof}


This lemma suggests that we introduce a relaxed version of $\mathcal{E}_e^{(m)}$:
\begin{align*}
  \hat{\mathcal{E}}_e^{(m)} &\coloneqq \gamma_1 \left\| \eEd{m}{} \right\|^2 + \left\| \fd_x \eEd{m}{} \right\|^2 + \gamma_2 \left( \left\| \eNd{m}{} \right\|^2 + \left\| \fd_x \eVd{m}{} \right\|^2 \right) + A^{(m)},
\end{align*}
where $\gamma_1, \gamma_2 >0$ are taken large enough in the next lemma so that the term $A^{(m)}$ is canceled out.

\begin{lemma} \label{lem:hEe:bound}
  For all $m$, $0 \leq \mathcal{E}_e^{(m)} \leq \hat{\mathcal{E}}_e^{(m)} $ holds for $\gamma_1 \geq C_4 + 1$ and $\gamma_2 \geq 5/2$.
\end{lemma}


\begin{proof}
  We see that
  \begin{align*}
    \hat{\mathcal{E}}_e^{(m)}
    &\geq \left( \gamma_1 - C_4 \right) \left\| \eEd{m}{} \right\|^2 + \left\| \fd_x \eEd{m}{} \right\|^2 + \left( \gamma_2 - \frac{3}{2} \right) \left( \left\| \eNd{m}{} \right\|^2 + \left\| \fd_x \eVd{m}{} \right\|^2 \right)\\
    &\geq \left\| \eEd{m}{} \right\|^2 + \left\| \fd_x \eEd{m}{} \right\|^2 + \left\| \eNd{m}{} \right\|^2 + \left\| \fd_x \eVd{m}{} \right\|^2 = \mathcal{E}_e^{(m)},
  \end{align*}
  which implies lemma.
\end{proof}


Now finally we are in the position to complete the proof of Theorem~\ref{thm_convergence}.
We have now
\begin{align*}
  \fd_t \hat{\mathcal{E}}_e^{(m)} &\leq \left( \gamma_1 C_1 + \gamma_2 C_2 + C_3 \right) \left( \mathcal{E}_e^{(m+1)} + \mathcal{E}_e^{(m)} \right) + \gamma' (c_0)^2 \left( (\Dt)^2 + (\Dx)^2 \right)^2\\
  &\leq \left( \gamma_1 C_1 + \gamma_2 C_2 + C_3 \right) \left( \hat{\mathcal{E}}_e^{(m+1)} + \hat{\mathcal{E}}_e^{(m)} \right) + \gamma' (c_0)^2 \left( (\Dt)^2 + (\Dx)^2 \right)^2
\end{align*}
for $\gamma' \coloneqq \gamma_1 + 2\gamma_2 + 4$. 
From Lemma~\ref{lem_gronwall}, we see that
\begin{align*}
  \mathcal{E}_e^{(m)} \leq \hat{\mathcal{E}}_e^{(m)} &\leq 2 \gamma' (c_0)^2T \left( (\Dt)^2 + (\Dx)^2 \right)^2 \mathrm{exp} \left( 4 \left( \gamma_1 C_1 + \gamma_2 C_2 + C_3 \right) T \right)\\
  &= C^2((\Dt)^2 + (\Dx)^2)^2
\end{align*}
holds for sufficiently small $\Dt, \Dx$ and some $C>0$.
This is the desired estimate.

\section{Numerical experiments}

In this chapter we compare the DVDM scheme in \eqref{sch_DVDM} with Glassey's scheme in \eqref{sch_Glassey} by numerical experiments.
Although we mainly use the notation in \cite{Payne}, we also need to deal with $V(t, x)$, and we need to make some additional definitions.


\subsection{Implementation of schemes}

\subsubsection{Glassey's scheme}

In the implementation of the Glassey's scheme \eqref{sch_Glassey}, the choice of $\Nd{1}{}$ is a challenge: the $\Nd{1}{}$ defined in Glassey's original paper \cite{Glassey} was of first-order precision. The $\Nd{1}{}$ defined by Payne in \cite{Payne} is of second-order precision, which seems reasonable at first glance: let us denote this case as (GP). However, (GP) is concerned with the effects of distortions in the energy. Therefore, we propose a new method for selecting $\Nd{1}{}$ that is less sensitive to distorted energy: let us denote this case as (GN).

The basic structure is as follows:

\begin{enumerate}
  \item $\Ed{0}{}, \Nd{0}{}$ and $(N_t(0,  k\Dx))_k$ are obtained by the initial conditions.
  \item $\Nd{1}{}$ is calculated (explained later).
  \item $\Ed{1}{}, \Nd{2}{}, \Ed{2}{}, \ldots$ are calculated alternately.
  \begin{enumerate}
    \item[(a)] If $\Ed{m}{}, \Nd{m}{}$ and $\Nd{m+1}{}$ are known, then, $\Ed{m+1}{}$ is calculated by \eqref{sch_Glassey1}.
    Here, considering the real part and the imaginary part, $2 \mathrm{i} I + \Dt \left\{ D_x^{\left\langle 2 \right\rangle} - \mathrm{diag} \left( \fa_t \Nd{m}{k} \right) \right\}$ is regarded as a $2K$-order sparse matrix consisting of $8K$ non-zero components.
    \item[(b)] If $\Ed{m}{}, \Nd{m}{}$ and $\Nd{m-1}{}$ are known, then, $\Nd{m+1}{}$ is calculated by \eqref{sch_Glassey2}, where $I - (\Dt)^2 D_x^{\left\langle 2 \right\rangle}/2$ is considered to be a $K$-order sparse matrix consisting of $3K$ nonzero components and independent of $m$.
    Therefore, by computing its LU decomposition in advance, $\Nd{m+1}{}$ can be computed faster.
  \end{enumerate}
\end{enumerate}
Since each matrix can be considered to be sparse, the computational complexity is estimated to be $\mathrm{O}(KM)$ for any $\Nd{1}{}$.

Here, in order to examine the accuracy of the conservation of energy, we define $\Vd{m}{} \coloneqq (D_x^{\left\langle 2 \right\rangle})^\dagger \fd_t \Nd{m}{}$, where $A^\dagger$ is the pseudo-inverse matrix of $A$.
This $\Vd{m}{}$ minimizes $\| \fd_t \Nd{m}{} - \cdd_x \Vd{m}{} \|$ but is not an exact solution of $\fd_t \Nd{m}{} = \cdd_x \Vd{m}{}$. This means $\mathcal{E}_G^{(m+1)}$ is not completely conserved.

Glassey's original study \cite{Glassey2} addresses this problem as follows: they determine $\Nd{1}{k}$ by $(N_\mathrm{G})_k \coloneqq \Nd{0}{k} + \Dt N^1(k\Dx)$ and further require the assumption
\begin{align}\label{eq_assumption2}
  \sum_{k=1}^K N^1(k\Dx) = 0.
\end{align}
This means
\begin{align}\label{eq_assumption3}
  \sum_{k=1}^K \fd_t \Nd{m}{k} = 0,
\end{align}
which allows them to determine $\Vd{0}{}$.
However, this $N_\mathrm{G}$ is first-order precision, which is inadequate, and the assumption \eqref{eq_assumption2} is unrealistic.

In \cite{Payne}, which proposed a scheme similar to Glassey's, $\Nd{1}{}$ with second-order accuracy was defined by Taylor expansion as
\begin{align*}
\hspace{-1mm} N(\Dt, k\Dx) &= N(0, k\Dx) + \Dt N_t(0,k\Dx) + \frac{1}{2}(\Dt)^2 N_{tt}(0,k\Dx) + \cdots\\
&\simeq \Nd{0}{k} + \Dt N^1(k\Dx) + \frac{1}{2} (\Dt)^2 \cdd_x \left( \Nd{0}{k} + \left\vert \Ed{0}{k} \right\vert^2 \right) \eqqcolon (N_\mathrm{GP})_k,
\end{align*}
which implies $(N_\mathrm{GP})_k = N(\Dt, k\Dx) + \mathrm{O}\left( (\Dt)^2 + (\Dx)^2 \right)$. However, this $N_\mathrm{GP}$ does not satisfy \eqref{eq_assumption3}. Instead, we define $\Nd{1}{k}$ by
\begin{align}\label{eq_N1def}
  (N_\mathrm{GN})_k \coloneqq (N_\mathrm{GP})_k - \frac{1}{L} \sum_{k=1}^K \left( (N_\mathrm{GP})_k - \Nd{0}{k} \right) \Dx,
\end{align}
so that \eqref{eq_assumption3} holds directly.

Here, $(N_\mathrm{GN})_k = N(\Dt, k\Dx) + \mathrm{O}\left( (\Dt)^2 + (\Dx)^2 \right)$ is shown as follows: since
\begin{align*}
  \frac{\rd}{\dt} \int_0^L N_{t} (t,x) \dx = \int_0^L N_{tt} (t,x) \dx = \int_0^L ( N + \vert E \vert^2 )_{xx} (t,x) \dx = 0,
\end{align*}
we have
\begin{align*}
  \sum_{k=1}^K \left( N(\Dt,k\Dx) - \Nd{0}{k} \right) \Dx &= \int_0^L N_t(\Dt/2, x) \dx + \mathrm{O}\left( (\Dt)^2 + (\Dx)^2 \right)\\
  &= \int_0^L N^1(x) \dx + \mathrm{O}\left( (\Dt)^2 + (\Dx)^2 \right)\\
  &= \mathrm{O}\left( (\Dt)^2 + (\Dx)^2 \right),
\end{align*}
which implies the conclusion.

\subsubsection{DVDM scheme}

In constructing the DVDM scheme, we have several choices: which to choose \eqref{sch_DVDM} or \eqref{sch_DVDM2}, and how to find approximate solutions of the nonlinear equations. We also have performed numerical experiments to determine which option is best, but we omit them for reasons of space. In the following, we describe the DVDM scheme when the simplified Newton method is used for \eqref{sch_DVDM2}.

The basic structure is as follows:

\begin{enumerate}
  \item $\Ed{0}{}, \Nd{0}{}$ and $\Vd{0}{}$ are obtained by the initial conditions.
  \item $\Ed{1}{}, \Nd{1}{}$ and $\Vd{1}{}$ are calculated by \eqref{sch_DVDM}.
  \item $\left( \Ed{m+1}{}, \Nd{m+1}{}\right)\, (m=1,2,\ldots)$ is calculated by \eqref{sch_DVDM2} in order.
  \item $\Vd{m}{} (m=2,3,\ldots)$ is calculated by the third equation in \eqref{sch_DVDM}.
\end{enumerate}

In the third step, we deal with nonlinear maps resulting from the definition of the DVDM scheme in \eqref{sch_DVDM2}, i.e., the $3K$-order nonlinear mapping $F_\mathrm{EN}^{(m)}(E, N)$ corresponding to \eqref{sch_DVDM2}, where the equation $F_\mathrm{EN}^{(m)} \left( \Ed{m+1}{}, \Nd{m+1}{} \right) = 0$ is equivalent to \eqref{sch_DVDM2}.
We write this in an abstract form as $F^{(m)}(x) = 0$, letting $F^{(m)} \coloneqq F_\mathrm{EN}^{(m)}$ and $x \coloneqq (E, N)$.\\

In order to solve this equation, 
we consider an iteration by $x^{(l+1)} \coloneqq x^{(l)} - r^{(l)}$ $(l=0,1,\ldots)$ for an initial point $x^{(0)}$.
The $r^{(l)}$ used in this case is assumed to satisfy 
$\nabla F^{(m)} \left( x^{(0)} \right) r^{(l)} = F^{(m)} \left( x^{(l)} \right)$. 
Notably, 
$\nabla F_\mathrm{EN}^{(m)}$ is considered to be a $3K$-order sparse matrix consisting of $19K$ nonzero components.

Here, the initial point $x^{(0)}$ of the iteration is defined as $\left( \Ed{m+1}{}, \Nd{m+1}{}\right)$ satisfying the Glassey's scheme in \eqref{sch_Glassey}. 
Although it is computationally time-consuming since it involves running Glassey's scheme in its entirety, this is chosen because of the stable and early convergence of the iterative method.

In addition, the stopping condition is $\left\| F^{(m)} \left( x^{(l)} \right) \right\| \leq \varepsilon$.
In the following, unless otherwise noted, $\varepsilon = 10^{-8}$ is fixed. Note that this means that discrete invariants are only preserved up to this accuracy.\\

As a whole, the computation amount of (D) is still estimated as $\mathrm{O}(KM)$, although it is definitely larger than that of (G).




\subsection{Solitary wave solution}

In order to describe a solitary wave solution of the Zakharov equations, we introduce Jacobi's elliptic function (shown in~\cite{Jacobi}, for example).

We assume that $q \in [0,1)$ and $u, \phi$ are satisfying
\begin{align*}
  u = \int_0^\phi \frac{\mathrm{d}\theta}{\sqrt{1-q\sin^2\theta}},
\end{align*}
then, the Jacobi's elliptic functions are defined as
\begin{align*}
  \sn (u, q) \coloneqq \sin \phi, \quad \cn (u,q) \coloneqq \cos \phi, \quad \dn (u,q) \coloneqq \sqrt{1-q\sin^2\phi}.
\end{align*}


\begin{lemma}
  The following properties hold.
  \begin{enumerate}
    \item We have $\sn (u,q)^2 + \cn (u,q)^2 = \dn (u,q)^2 + q \sn (u,q)^2 = 1$.
    \item We have $\dfrac{\partial}{\partial u} \sn (u,q) = \cn (u,q) \dn (u,q)$, $\dfrac{\partial}{\partial u} \cn (u,q) = - \sn(u,q) \dn(u,q)$ and $\dfrac{\partial}{\partial u} \dn (u,q) = -q \sn (u,q) \cn (u,q)$.
    \item We have $\dn ( u + 2K(q), q) = \dn (u,q)$, where ${\displaystyle K(q) \coloneqq \int_0^{\pi/2} \dfrac{\mathrm{d}\theta}{\sqrt{1 - q\sin^2\theta}}}$ is the complete elliptic integral of the first kind.
  \end{enumerate}
\end{lemma}

Using these definitions, the solitary wave solution satisfying the Zakharov equations \eqref{eq_Zak} is written as follows:
\begin{align}\label{eq_soliton1}
  \begin{cases}
    E(t, x) = \Emax \dn \left( \dfrac{\Emax}{\sqrt{2(1-v^2)}}(x-vt), q \right) \exp \left( \mathrm{i} \phi (x - ut) \right),\\
    N(t, x) = -\dfrac{\Emax^2}{1-v^2} \dn \left( \dfrac{\Emax}{\sqrt{2(1-v^2)}}(x-vt), q \right)^2 + N_0,
  \end{cases}
\end{align}
where $\Emax, v, q, \phi, u, N_0$ (and period $L$) are the parameters.
Notably, the condition for \eqref{eq_soliton1} to satisfy the Zakharov equations \eqref{eq_Zak} is
\begin{align*}
  \phi = \frac{v}{2},\quad u = \frac{v}{2} + \frac{2N_0}{v} - \frac{ 2-q }{ v(1-v^2) } \Emax^2,
\end{align*}
and the condition for \eqref{eq_soliton1} to satisfy the periodic boundary condition \eqref{eq_period} is
\begin{align*}
  L = \frac{2 \sqrt{2(1-v^2)}}{\Emax} K(q),\quad \frac{v}{2} L = 2 \pi m\quad (m \in \mathbb{Z}),
\end{align*}
where $m$ can be an integer, but in the following it is unified as $m=1$, unless otherwise noted.
At this point, $\Emax, N_0$ and $L$ are left as undetermined parameters.

While the above discussion is based on \cite{Payne}, in this study we also need to consider the definition of $V$ corresponding to this solitary wave solution.


\begin{lemma}\label{lem_solitonV}
  Let $E(t,x), N(t,x)$ be a solution of \eqref{eq_Zak} satisfying \eqref{eq_soliton1}.
  Then, $V(t,x)$ satisfying $N_t =V_{xx},\, V_t =N+\vert E \vert^2$ is described as follows:
  \begin{align}\label{eq_soliton2}
    V(t, x) = \dfrac{\sqrt{2} v \Emax}{\sqrt{1-v^2}} E_2(\varphi_V(x-vt), q) - \frac{N_0}{v} (x-vt) + V_0,
  \end{align}
  where $V_0$ is a integral constant, ${\displaystyle E_2(\varphi, q) \coloneqq \int_0^\varphi \sqrt{1-q\sin^2\phi}\mathrm{d}\phi}$ is the incomplete elliptic integral of the second kind, and
  \begin{align*}
    \varphi_V(x) \coloneqq l\pi + \sin^{-1} \left( \sn \left( \dfrac{\Emax}{\sqrt{2(1-v^2)}}x -2lK(q)\right),q \right)
  \end{align*}
  for $(l-1/2)L < x \leq (l+1/2)L$, $l \in \mathbb{Z}$.
\end{lemma}




In order for $V(t, x)$ to satisfy the periodic boundary condition at the time $t=0$, we needs
\begin{align}\label{eq_N0}
  N_0 = \frac{2\sqrt{2}v^2 \Emax}{L\sqrt{1-v^2}} E_2 \left( \frac{\pi}{2},q \right).
\end{align}
However, in the previous studies \cite{Payne, Glassey2}, the parameter $N_0$ is defined as
\begin{align*}
  N_0 = \frac{2\sqrt{2} \Emax}{L\sqrt{1-v^2}} E_2 \left( \frac{\pi}{2}, q \right)
\end{align*}
so that $\int_0^L N(t, x)\dx = 0$ is satisfied, which is inconsistent with \eqref{eq_N0}.
How this discrepancy should be interpreted is very important, but at least in this study, we adopt \eqref{eq_N0} because the periodicity of $V(t,x)$ is extremely important for the DVDM scheme.
In this case, if we choose parameters $\Emax$ and $L$, the other parameters are determined from the above conditions.


\subsection{Experiment 1 : Single solitary wave solution case}

In this section, we run Glassey's scheme and the DVDM scheme with the single solitary wave solution described in the previous section.
The parameter $L$ is fixed to $L=20$ and for the other parameter $\Emax$ we consider several values (the case $\Emax=1$ corresponds to Glassey's previous study \cite{Glassey2}).
For each $\Emax$, the temporal and spatial mesh sizes are halved as $\Dt = \Dx = 0.1, 0.05, 0.025, \ldots$.

The goal time $T$ is mainly set to the time $T_L\coloneqq L/v$, which is the time until the solitary wave comes around once in the spatial domain.
However, when $\Emax$ is large, the phase velocity $\phi$ of $E$ is very large and convergence becomes extremely slow. In that case, the time $T_1 \coloneqq 1/v = T_L/L$ for the solitary wave to advance just 1 may be chosen instead.

The results of the numerical experiments are evaluated by the following indices: the initial value of energy $\mathcal{E}^{(0)} = \mathcal{E}_\mathrm{G}^{(1)}$ or $\mathcal{E}_\mathrm{D}^{(0)}$, the maximum error of energy $d\mathcal{E} \coloneqq \max_m \vert \mathcal{E}_\mathrm{G}^{(m+1)} - \mathcal{E}_\mathrm{G}^{(1)} \vert$ or $\max_m \vert \mathcal{E}_\mathrm{D}^{(m)} - \mathcal{E}_\mathrm{D}^{(0)} \vert$, the maximum error of each component $\varepsilon_E \coloneqq \max_m \| \eEd{m}{} \|$ and $\varepsilon_N$, and the run time.

The execution environment is Macbook Air (M1, 2020, 8GB), Python 3.10.4. The actual code used is shown in \url{https://github.com/kawai-sk/Zakharov}.


\subsubsection{One cycle case}

First, we show the results when $\Emax = 1,5$ and $T=T_L$ in Table \ref{table_exp1-1}.

\begin{table}[hptb]
  \caption{Results of Experiment 1, one cycle case}
  \label{table_exp1-1}
  \centering
  \begin{tabular}{ l l l l l l l l }
    \hline
    $E_{\mathrm{max}}$ & $\Dt$ & & $\mathcal{E}^{(0)}$ & $d\mathcal{E}$ & $\varepsilon_E$ & $\varepsilon_N$ & time \\ \hline



    1 & 0.1 & (GP) & 1.01 & e-6 & 9.3e-2 & 9.8e-2 & 0.254s \\
    & & (GN) & 1.01 & e-13 & 8.4e-2 & 9.8e-2 & 0.256s \\
    & & (D) & 1.01 & e-8 & 2.1e-2 & 4.3e-2 & 2.74s \\ \cline{2-8}

    & 0.05 & (GP) & 1.01 & e-8 & 2.2e-2 & 2.4e-2 & 0.675s \\
    & & (GN) & 1.01 & e-12 & 2.1e-2 & 2.4e-2 & 0.685s \\
    & & (D) & 1.00 & e-11 & 5.3e-3 & 1.0e-2 & 8.06s \\ \hline










    5 & 0.1 & (GP) & 125.41 & 6.40 & 6.59 & 31.2 & 0.247s \\
    & & (GN) & 125.20 & e-11 & 6.59 & 31.1 & 0.256s \\
    & & (D) & 115.27 & e-5 & 6.62 & 30.8 & 3.51s \\ \cline{2-8}

    & 0.05 & (GP) & 120.00 & 0.12 & 6.63 & 27.0 & 0.691s \\
    & & (GN) & 129.98 & e-11 & 6.63 & 27.0 & 0.691s \\
    & & (D) & 117.62 & e-4 & 6.63 & 12.9 & 7.98s \\ \cline{2-8}

    & 0.025 & (GP) & 118.81 & e-3 & 6.63 & 7.69 & 2.12s \\
    & & (GN) & 118.81 & e-10 & 6.63 & 7.69 & 2.11s \\
    & & (D) & 118.22 & e-5 & 6.63 & 3.09 & 28.2s \\ \cline{2-8}

    & 0.0125 & (GP) & 118.52 & e-5 & 6.59 & 1.88 & 7.07s \\
    & & (GP) & 118.52 & e-10 & 6.60 & 1.88 & 7.05s \\
    & & (D) & 118.38 & e-6 & 6.63 & 0.76 & 84.0s \\ \hline
  \end{tabular}
\end{table}

We can see that, when $\Emax$ (and $\phi$) is not sufficiently small, the quadratic convergence of the error can be observed only for small $\Dt$'s.
Note that, even when $ \Dt $ is not sufficiently small, the error with respect to $E$ is bounded because $\left\| \Ed{m}{} \right\|_\infty$ is bounded as shown in Lemma~\ref{lem_sup}.

The following is a detailed description of each item:
\begin{itemize}
  \item The accuracy of energy conservation in (GP) decreases as $\Emax$ increases, while that of (D) or (GN) remain sufficiently accurate. Except for this large difference and the small difference in $\varepsilon_E$, (GP) and (GN) exhibit roughly the same behavior (they are sometimes collectively denoted as (G)).
  \item For any $\Emax$, the execution time of (G) is constant and more than 10 times shorter than that of (D) for the same $\Dt$.
  \item The error with respect to the same $\Dt$ is directly proportional to $\Emax$.
  Although (D) shows a slight advantage with respect to approximation accuracy, it is not enough to overcome the disadvantage in execution time.
\end{itemize}



\subsubsection{Short time case}

Next, we show the results when $\Emax = 5$ and $T = T_1$ in Table \ref{table_exp1-3}.
To distinguish it from the case $T = T_L$, the maximum error in the case $T=T_1$ is denoted as $\varepsilon'_E$ and $\varepsilon'_N$ instead of $\varepsilon_E$ and $\varepsilon_N$.

\begin{table}[htb]
  \caption{Results of Experiment 1, short time case}
  \label{table_exp1-3}
  \centering
  \begin{tabular}{ l l l l l l l l }
    \hline
    $E_{\mathrm{max}}$ & $\Dt$ & & $\mathcal{E}^{(0)}$ & $d\mathcal{E}$ & $\varepsilon'_E$ & $\varepsilon'_N$ & time \\ \hline
    5 & 0.1 & (GP) & 126.1 & 0.017 & 6.58 & 14.53 & 0.0137s \\
    & & (GN) & 126.1 & e-12 & 6.58 & 14.53 & 0.0183s \\
    & & (D) & 115.3 & e-6 & 6.63 & 7.21 & 0.157s \\ \cline{2-8}

    & 0.05 & (GP) & 120.1 & e-4 & 5.98 & 3.19 & 0.0363s \\
    & & (GN) & 120.1 & e-11 & 5.99 & 3.19 & 0.0420s \\
    & & (D) & 117.6 & e-6 & 6.32 & 1.08 & 0.506s \\ \cline{2-8}

    & 0.025 & (GP) & 118.8 & e-6 & 1.94 & 0.64 & 0.124s \\
    & & (GN) & 118.8 & e-11 & 1.94 & 0.64 & 0.124s \\
    & & (D) & 118.2 & e-7 & 2.21  & 0.24 & 1.15s \\\cline{2-8}

    & 0.0125 & (GP) & 118.52 & e-8 & 0.49 & 0.15 & 0.398s \\
    & & (GN) & 118.52 & e-11 & 0.50 & 0.15 & 0.405s \\
    & & (D) & 118.37 & e-7 & 0.57 & 0.06 & 4.43s \\ \hline 







  \end{tabular}
\end{table}

We can confirm that $E$ is also convergent in the short time. Basically, the same trend is observed as in the $T=T_L$ case, but the energy conservation accuracy of (GP) is noteworthy.
While the execution time and the error of each scheme varies linearly for $T=T_L$, less than about $L=20$ times that for $T=T_1$, this is not the case only for the conservation of energy in (GP).
For example, when $\Dt = 0.1$, $d\mathcal{E}$ for $T=T_L$ is more than 500 times larger than for $T=T_1$.
This suggests that the conservation accuracy of energy in (GP) deteriorates more rapidly than linearly.

\subsubsection{Long time case}

Based on the previous results, we are concerned that the accuracy of the energy conservation in (GP) seems to be unstable with respect to time.
To gain more insight into this point, we present the results of an experiment conducted at time $T=20T_L$, where the single wave soliton solution comes around 20 periods.

Here, the graph (Figure \ref{figure_exp3}) shows the evolution of the error versus time for $\Emax=2, \Dt=0.025$. This would be a notable example in several respects.

\begin{figure}[htbp]

  \centering
  \includegraphics[width=12.5cm]{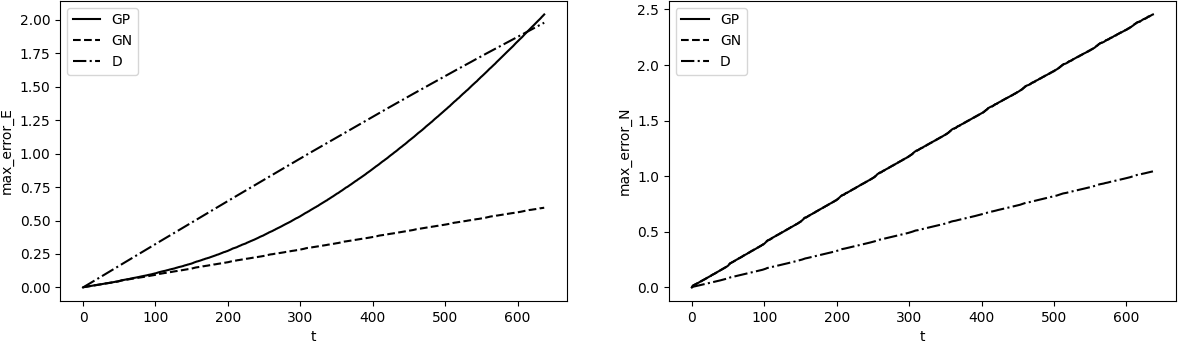}
  \caption{The errors for $\Emax=2, \Dt=0.025, T=20T_L$}
  \label{figure_exp3}
\end{figure}

The conservation accuracy of energy in (GP) deteriorates with time, so the error of $E$ increases nonlinearly. As a result, (GP) can be less favorable than (D) in long-time experiments. This means that the energy distortion has a negative effect on the (GP) proposed in previous studies. However, our proposed (GN) overcomes this drawback.


\subsubsection{Summary of this section}

As theoretically shown by Theorem \ref{thm_convergence}, the quadratic convergence is numerically demonstrated for many $\Emax$.
For the same $\Dt = \Dx$, there is no significant difference between (D) and (G) in accuracy for $E$ and $N$.
In terms of execution time, however, (G) is by far superior.
Instability due to distortion of conserved quantities corresponding to energy is a concern of (GP), but this issue has been resolved in our proposal (GN).
Therefore, as an overall evaluation, (GN) is generally considered superior.

\subsection{Experiment 2 : The collision of two solitary waves}

Now, we consider a more realistic situation in which two solitary waves collide.
In this case, as in the previous studies \cite{Chang, Payne, Glassey2}, we consider two solutions that are determined for the same $\Emax$ and $L=20$ in the long computational interval $8L=160$.
We assume that they have opposite velocities (i.e., one has $m = 1$ and the other has $m = -1$).

In some previous studies, the result for very small $\Dt$ is considered as ``exact solution'' and the errors from the numerical solutions are evaluated.
In this study as well, we consider the case $\Dt=0.0125$ for (GN) as the ``exact solution''.
Since it is difficult to make $\Dt$ even smaller due to the execution environment, we treat only $\Emax = 1, 2$.

The major difference from the previous study is that we must also deal with $V(t, x)$ to perform (D), which requires a deeper consideration of the initial conditions.


\subsubsection{Initial condition}

Let $E_\pm(t, x), N_\pm(t,x)$ and $V_\pm(t,x)$ be a single solitary wave solution that can be written in the form \eqref{eq_soliton1}, \eqref{eq_soliton2} for some $\Emax$, $L=20$ and $m=\pm 1$, where the solution for $m=1 (-1)$ moves in the x-axis positive (negative) direction.
Then, the initial condition for $E$ in this soliton collision experiment is naturally defined as follows:

\begin{align}\label{eq_collision1}
  E_0(x) \coloneqq E(0, x) =
  \begin{cases}
    E_+(0, -L/2) = 0 &\text{if } x\in [0, 3L)\\
    E_+(0, x - 7L/2) &\text{if } x \in [3L, 4L)\\
    E_+(0, L/2) = E_-(0, -L/2) = 0 &\text{if } x=4L\\
    E_-(0, x - 9L/2) &\text{if } x \in (4L, 5L)\\
    E_-(0, L/2) = 0 &\text{if } x \in [5L, 8L)\\
  \end{cases}
  .
\end{align}
Similar definitions hold for $N_0(x) \coloneqq N(0,x)$ and $(N_t)_0(x) \coloneqq N_t(0,x)$, which satisfy the periodic boundary condition of the length $8L$.
However, if we denote $V_1(x) \coloneqq V(0,x)$ for a similarly defined initial condition of $V$, this leads to a contradiction.
In other words, $(N_t)_0 \neq (V_1)_{xx}$ holds due to the non-differentiability of $V_1$.

A way to circumvent this difficulty is to employ $V(0, x)$ corresponding to $(N_t)_0$ via the following lemma.


\begin{lemma}
  Let $E_0(x), N_0(x)$ and $(N_t)_0(x)$ be the initial solution of \eqref{eq_Zak} satisfying \eqref{eq_collision1} and so on.
  In addition, we assume that $V(t, x)$ satisfies
  \begin{align*}
    V(t, x) = V_+(x - vt) + V_-(x + vt)
  \end{align*}
   in a neighborhood of $t = 0$.
   Then, $V_0(x) \coloneqq V(0, x)$ satisfying $(N_t)_0(x) = (V_0)_{xx}(x)$ and $V_t(0, x) = N_0(x) + \left\vert E_0(x) \right\vert^2$ can be written as follows:
  \begin{align}\label{eq_solitonV2}
    V_0(x) \coloneqq
    \begin{cases}
      0 &\text{if } x\in [0, 3L)\\
      \dfrac{\sqrt{2} v \Emax}{\sqrt{1-v^2}} E(\varphi_V(x-7L/2), q) + \dfrac{N_0}{2v}L &\text{if } x \in [3L, 4L)\\
      -\dfrac{\sqrt{2} v \Emax}{\sqrt{1-v^2}} E(\varphi_V(x-9L/2), q) + \dfrac{N_0}{2v}L &\text{if } x \in [4L, 5L)\\
      0 &\text{if } x \in [5L, 8L)\\
    \end{cases},
  \end{align}
  where $\varphi_V(x)$ is defined in Lemma \ref{lem_solitonV}.
\end{lemma}


Another remedy is to consider $(N_t)_1(x) \coloneqq N_t(0,x) = (V_1)_{xx}(x)$ as the initial condition for $N_t$.

It is important to note that the initial conditions of $N_t$ and $V$ have a significant influence on the results of the numerical experiments.
For $i=0,1$, let us denote the scheme (G) based on $(N_t)_i$ as ($\mathrm{G}_i$), and  also (D) based on $V_i$ as ($\mathrm{D}_i$).
Under these settings, let us now consider the case where $\Emax$ is small, such as $\Emax=0.5$, where clearly early convergence can be observed for any method.
Then, ($\mathrm{G}_0$) and ($\mathrm{D}_0$) (and ($\mathrm{G}_1$) and ($\mathrm{D}_1$)) converge to similar waveforms, but the results of ($\mathrm{G}_0$) and ($\mathrm{G}_1$) (and similarly ($\mathrm{D}_0$) and ($\mathrm{D}_1$)) are clearly different.
This seems natural since the collision problem determined by $(N_t)_0$ and $V_0$ is different from that determined by $(N_t)_1$ and $V_1$.
In the numerical comparisons, these differences can cause serious confusion, especially when $\Emax$ is large and convergence can be expected only for small $\Dt$'s.
There it is hard to know whether the remaining difference (of solutions) comes from the choice of initial data or the difference of schemes.
Therefore, in the following, we adopt $(N_t)_0$ and $V_0$, which seem to be more natural initial conditions for the soliton collision experiment.
Since it seems mathematically very difficult to obtain an exact analytical solution for the situation where two solitons collide, we does not pursue it in this study.


\subsubsection{Result}

Now, we show the results when $\Emax = 1,2$ and $T=T_L$ in Table \ref{table_exp2}.
We denote the error with the ``exact solution'' as $\tilde{\varepsilon}_E, \tilde{\varepsilon}_N$ to distinguish it from the case where the error with the true exact solution is evaluated.

\begin{table}[htb]
  \caption{Results of Experiment 2}
  \label{table_exp2}
  \centering
  \begin{tabular}{ l l l l l l l l l}
    \hline
    $E_{\mathrm{max}}$ & $\Dt$ & & $\mathcal{E}^{(0)}$ & $d\mathcal{E}$ & $\tilde{\varepsilon}_E$ & $\tilde{\varepsilon}_N$ & time \\ \hline\hline
    1 & 0.1 & (GP) & 2.59 & 9e-4 & 1.06 & 0.42 & 0.922s \\
    &  & (GN) & 2.59 & e-12 & 1.06 & 0.42 & 0.992s \\
    & & (D) & 2.59 & e-7 & 0.29 & 0.18 & 10.3s \\ \cline{2-8}

    & 0.05 & (GP) & 2.59 & 1e-4 & 0.28 & 0.10  & 3.40s \\
    & & (GN) & 2.59 & e-11 & 0.28 & 0.10  & 3.40s \\
    & & (D) & 2.59 & e-9 & 0.081 & 0.035 & 43.2s \\ \hline

    2 & 0.1 & (GP) & 17.44 & 0.042 & 5.58 & 4.42 & 0.687s \\
    & & (GN) & 17.44 & e-11 & 5.59 & 4.42 & 0.604s \\
    & & (D) & 17.22 & e-8 & 4.67 & 1.52 & 13.3s \\ \cline{2-8}

    & 0.05 & (GP) & 17.33 & 5e-3 & 1.81 & 0.82  & 2.20s \\
    & & (GN) & 17.33 & e-10 & 1.79 & 0.82  & 2.13s \\
    & & (D) & 17.28 & e-6 & 1.24 & 0.39 & 41.2s \\ \cline{2-8}

    & 0.025 & (GP) & 17.30 & 7e-4 & 0.38 & 0.20 & 8.39s \\
    & & (GN) & 17.30 & e-9 & 0.38 & 0.20 & 8.64s \\
    & & (D) & 17.29 & e-8 & 0.30 & 0.052\textit{} & 200s \\ \hline 




  \end{tabular}
\end{table}

The basic trend of the results is the same as in Experiment 1, so (G) is superior for the same reasons as in Experiment 1.





\section{Concluding remarks}

We have presented a mathematical analysis for the energy-conservative scheme for the Zakharov equations based on DVDM.
There we proved solvability under the loosest possible assumptions and improved the argument of convergence by Glassey (the unnecessary condition $\Dt = \Dx$ is removed, at least in the discussion in the convergence estimate).
We also reorganized the discussion (borrowed from Glassey's argument) so that the strategy is more visible.
In other words, we now have a better understanding of how to utilize invariants in the analysis of conservative schemes, which we expect can be applied to other schemes.

In addition, the DVDM scheme is compared with Glassey's scheme by extensive numerical experiments.
The results show that (G) outperforms (D) in most situations.
This is because the DVDM scheme is merely a consequence of the general method, whereas Glassey's scheme is an implicit linear scheme cleverly constructed utilizing the properties of the Zakharov equations (and thus faster to solve).
However, we also found that if the scheme proposed by Glassey in \cite{Glassey} is implemented as (GP), it becomes unstable in longer experiments due to the distortion of the conserved quantities. This problem is successfully solved by our proposed method of determining $\Nd{1}{}$.

For future prospects, we are exploring the possibility of applying the structure of the convergence argument to other conservative schemes for other equations, and we are also interested in the other conservative schemes for the Zakharov equations not dealt with in this study.
As mentioned in the introduction, the energy-conservative schemes in \cite{Chang, Ji, Pan} also all have a distorted form of the invariant corresponding to energy with respect to time.
It is a point of concern whether this still implies instability for long-time calculations and whether the issue can be addressed.

\section*{Statements and Declarations}
There are no conflicts of interests that are related to the content of this study.

\end{document}